\documentclass[a4paper,11pt]{article}

\usepackage{authblk}
\usepackage{parskip}
\usepackage{mathtools}
\usepackage{amssymb}
\usepackage{amsmath}
\usepackage{latexsym}
\usepackage{mathrsfs}  
\usepackage{bbm}       
\usepackage{amsthm}    
\usepackage{relsize}   
\usepackage{graphicx}
\usepackage{thmtools}  
\usepackage{mathrsfs}  
\usepackage{paralist}  
\usepackage{lipsum}    
\usepackage[all]{xy}   

\numberwithin{equation}{section}

\usepackage{titlesec}   

\titleformat{\section}[block]{\bfseries\filcenter}
{{\upshape\thesection\enspace}}{.5em}{}

\titleformat{\subsection}[block]{\filcenter}
{{\upshape\thesubsection\enspace}}{.5em}{} 

\usepackage{enumitem}  
\setlist{nosep}  
\setitemize[0]{leftmargin=*}
\setenumerate[0]{leftmargin=*}
\setenumerate[1]{label={(\arabic*)}} 

\newcommand{\N}{\mathbb{N}}     
\newcommand{\Q}{\mathbb{Q}}     
\newcommand{\I}{\mathbb{I}}     
\newcommand{\R}{\mathbb{R}}     
\newcommand{\C}{\mathbb{C}}     
\newcommand{\Prob}{\mathbb{P}}  
\newcommand{\Exp}{\mathbb{E}}   
\newcommand{\st}{\,:\,}         
\newcommand{\goth}[1]{\mathfrak{#1}} 
\newcommand{\ind}[2]{\mathbbm{1}_{#1}\left( #2 \right)}          
\newcommand{\inner}[2]{\left\langle #1 \, , \, #2 \right\rangle} 
\newcommand{\norm}[1]{\left|\left|#1\right|\right|}              
\newcommand{\triplet}[3]{\left( #1, #2, #3 \right) }             
\newcommand{\ProbSpace}{\triplet{\Omega}{\mathscr{F}}{\Prob}}    
\newcommand{\abs}[1]{\left| #1 \right|}                          
\renewcommand{\qedsymbol}{$\square$}                       
\newcommand{\Levy}{L\'{e}vy} 
\newcommand{\cadlag}{c\`{a}dl\`{a}g }
\newcommand{\Ito}{It\^{o}}
\newcommand{\defeq}{\mathrel{\mathop:}=}                         
\newcommand\restr[2]{{
  \left.\kern-\nulldelimiterspace 
  #1 
  \vphantom{\big|} 
  \right|_{#2} 
  }}
  
\theoremstyle{plain} 
\newtheorem{theo}{Theorem}[section]    
\newtheorem{prop}[theo]{Proposition} 
\newtheorem{coro}[theo]{Corollary}
\newtheorem{lemm}[theo]{Lemma}
\newtheorem{assu}[theo]{Assumption}
\newtheorem{rema}[theo]{Remark}

\theoremstyle{definition} 
\newtheorem{defi}[theo]{Definition}

\newtheorem{nota}[theo]{Notation}

\declaretheoremstyle[%
  spaceabove=-5pt,%
  spacebelow=6pt,%
  headfont=\normalfont\itshape,%
  postheadspace=1em,%
  qed=\qedsymbol%
]{mystyle} 
\declaretheorem[name={Proof},style=mystyle,unnumbered,
]{prf}

 \title{L\'{e}vy Processes and Infinitely Divisible Measures in the Dual of a Nuclear Space.}
\author{C. A. Fonseca-Mora}
\affil{  Escuela de Matem\'{a}tica, Universidad de Costa Rica, San Jos\'{e}, 11501-2060, Costa Rica. \\

\noindent E-mail:  christianandres.fonseca@ucr.ac.cr }

\date{}

\begin{document}

 \maketitle

\abstract{Let $\Phi$ be a nuclear space and let $\Phi'_{\beta}$ denote its strong dual. In this work we prove the existence of c\`{a}dl\`{a}g versions, the L\'{e}vy-It\^{o} decomposition, and the L\'{e}vy-Khintchine formula for $\Phi'_{\beta}$-valued L\'{e}vy processes. Moreover, we give a characterization for L\'{e}vy measures on $\Phi'_{\beta}$ and provide conditions for the existence of regular versions to cylindrical L\'{e}vy processes in $\Phi'$. Furthermore, under the assumption that $\Phi$ is a barrelled nuclear space we establish a one-to-one correspondence between infinitely divisible measures on $\Phi'_{\beta}$ and L\'{e}vy processes in $\Phi'_{\beta}$. Finally, we prove the L\'{e}vy-Khintchine formula for infinitely divisible measures on $\Phi'_{\beta}$.}

\smallskip



\section{Introduction}

This work is concerned with the study of properties of L\'{e}vy processes (i.e. processes with independent and stationary increments) which take values in the strong dual $\Phi'_{\beta}$ of a (general) nuclear space $\Phi$. In particular, we are interested in the validity in the context of the dual of a nuclear space of some properties of the L\'{e}vy processes in $\R^{d}$ such as the existence of c\`{a}dl\`{a}g versions, the  L\'{e}vy-It\^{o} decomposition and the L\'{e}vy-Khintchine formula. We  are also interested in finding characterizations for the L\'{e}vy measures on $\Phi'_{\beta}$ and conditions for the existence of regular versions of cylindrical L\'{e}vy processes in $\Phi'$. Moreover, we intend to find conditions on the nuclear space $\Phi$ in order to guarantee the existence of a one-to-one correspondence between infinitely divisible measures on $\Phi'_{\beta}$ and L\'{e}vy processes taking values in $\Phi'_{\beta}$.   

In the case of the dual of a nuclear space, the study of L\'{e}vy processes has been mostly concentrated on the study of Wiener processes and stochastic analysis defined with respect to these processes (see e.g. \cite{BojdeckiJakubowski:1990, Ito, KallianpurXiong}). Indeed, to the extent of our knowledge the only previous works on the study of some properties of general L\'{e}vy processes in the dual of specific classes of nuclear spaces are \cite{Baumgartner:2015, PerezAbreuRochaArteagaTudor:2005, Ustunel:1984}.

On the other hand, although properties of infinitely divisible measures defined on locally convex spaces were explored by several authors (see e.g. \cite{Dettweiler:1976, Fernique:1967, Tortrat:1967, Tortrat:1969}), we are not aware of any previous work that studies the correspondence between L\'{e}vy processes and infinitely divisible measures on the dual of general nuclear spaces. In this article we prove this correspondence and the L\'{e}vy-Khintchine formula for infinitely divisible measures on $\Phi'_{\beta}$.

We start in Section \ref{sectionPrelim} with some preliminary results on nuclear spaces, cylindrical and stochastic processes and Radon measures on the dual of a nuclear space. Then in Section \ref{subSectionInfDivMeas} we use some results of Siebert \cite{Siebert:1974, Siebert:1976} to study the problem of embedding a given infinite divisible measure $\mu$ into a continuous convolution semigroup of  probability measures $\{ \mu_{t} \}_{t \geq 0}$ on $\Phi'_{\beta}$. Later, in Section \ref{subSectionLPCLP} by using the results in \cite{FonsecaMora:2018} that gives conditions for a cylindrical process to have a c\`{a}dl\`{a}g version (known as regularization theorems), we give conditions for the existence of a c\`{a}dl\`{a}g L\'{e}vy version of a given cylindrical L\'{e}vy process in $\Phi'$ or of a $\Phi'_{\beta}$-valued L\'{e}vy process. In particular we show that if the space $\Phi$ is nuclear and barrelled, then every L\'{e}vy process in $\Phi'_{\beta}$ has a c\`{a}dl\`{a}g version that is also a L\'{e}vy process. 
 
In Section \ref{subsectionCLPIDM} we proceed to prove the one-to-one correspondence between L\'{e}vy processes and infinitely divisible measures on $\Phi'_{\beta}$. Here it is important to remark that the standard argument to prove the correspondence that works in finite dimensions (see e.g. Chapter 2 in \cite{Sato}) does not work in our context as the Kolmogorov extension theorem is not applicable on the dual of a general nuclear space. To overcome this obstacle we use a projective system version of the Kolmogorov extension theorem (see \cite{RaoSPGT}, Theorem 1.3.4) to show a general theorem (that holds for Hausdorff locally convex spaces) that guarantees the existence of a cylindrical L\'{e}vy process $L$ whose cylindrical distributions extends for each time $t$ to the measure $\mu_{t}$ of the continuous convolution semigroup $\{ \mu_{t}\}_{t\geq 0}$ in which the given infinitely divisible measure $\mu$ can be embedded. Then, for this cylindrical process $L$ we use the results in Section \ref{subSectionLPCLP} to show the existence of a $\Phi'_{\beta}$-valued c\`{a}dl\`{a}g L\'{e}vy process $\tilde{L}$  that is a version of $L$, and hence the probability distribution of $\tilde{L}_{1}$ coincides with $\mu$ and then we have the correspondence. In Section \ref{subsectionWCPP} we review some properties of Wiener processes in $\Phi'_{\beta}$.   

After we study in Sections \ref{subsectionPRMPI} and  \ref{subSectionPMPILP} the basic properties of Poisson integrals defined by Poisson random measures on $\Phi'_{\beta}$, in Sections \ref{subSectionLM} and \ref{subSectionLMLP} we investigate the properties of the L\'{e}vy measures on $\Phi'_{\beta}$. In particular, we will show that a Borel measure $\nu$ on $\Phi'_{\beta}$ with $\nu(\{0\})=0$ is a L\'{e}vy measure on $\Phi'_{\beta}$ if and only if $\nu$ is a bounded Radon measure when restricted to the complement of every neighborhood of zero and if for the norm $\rho'$ of some Hilbert space $\Phi'_{\rho}$ continuously embedded in the dual space  $\Phi'_{\beta}$, we have   
$\int_{B_{\rho'}(1)} \rho'(f)^{2} \nu (df) < \infty$ and that $\nu$ is a bounded Radon measure when restricted to the complement of the unit ball $B_{\rho'}(1)$ of $\rho'$. Within the context of the dual of a nuclear space, the above characterization generalizes the characterization of L\'{e}vy measures obtained by Dettweiler in \cite{Dettweiler:1976} for the case of complete Badrikian spaces. 

Later, we proceed to prove in Section \ref{subsectionLID} the so-called L\'{e}vy-It\^{o} decomposition for the paths of a $\Phi'_{\beta}$-valued L\'{e}vy process. More specifically, we show that a $\Phi'_{\beta}$-valued L\'{e}vy process $L=\{ L_{t} \}_{t \geq 0}$ has a decomposition of the form (see Theorem \ref{levyItoDecompositionTheorem}):
$$L_{t}=t\goth{m}+W_{t}+\int_{B_{\rho'}(1)} f \widetilde{N} (t,df)+\int_{B_{\rho'}(1)^{c}} f N (t,df), \quad \forall t \geq 0,$$
where $\goth{m} \in \Phi'_{\beta}$, $\rho'$ is the norm associated to the L\'{e}vy measure $\nu$ of $L$ as described in the previous paragraph, $\{ W_{t} \}_{t \geq 0}$ is a Wiener process taking values in a Hilbert space continuously embedded in the dual space  $\Phi'_{\beta}$, $\left\{ \int_{B_{\rho'}(1)} f \widetilde{N} (t,df): t\geq 0 \right\}$ is a mean-zero, square integrable, c\`{a}dl\`{a}g L\'{e}vy process taking values in a Hilbert space continuously embedded in the dual space  $\Phi'_{\beta}$ and $\left\{ \int_{B_{\rho'}(1)^{c}} f N (t,df): t\geq 0 \right\}$ is a $\Phi'_{\beta}$-valued c\`{a}dl\`{a}g L\'{e}vy process defined by means of a Poisson integral with respect to the Poisson random measure $N$ of the L\'{e}vy process $L$.

Our L\'{e}vy-It\^{o} decomposition improves the decomposition proved by \"{U}st\"{u}nel in \cite{Ustunel:1984} in two directions. First, compared with the result in \cite{Ustunel:1984}, we have obtained a much simpler and more detailed characterization of the components of the decomposition. In particular, contrary to the decomposition in \cite{Ustunel:1984} we have been able to show the independence of all the random components in our decomposition. Second, for our decomposition we only assume that the space $\Phi$ is nuclear and we do not assume any property on the dual space $\Phi'_{\beta}$; this is in contrast to the decomposition in \cite{Ustunel:1984} where $\Phi$ is assumed to be separable, complete,  barreled, and nuclear, and $\Phi'_{\beta}$ is assumed to be  Suslin and nuclear. The same assumptions on $\Phi$ and $\Phi'_{\beta}$ are required for the decomposition obtained by Baumgartner in \cite{Baumgartner:2015} (at least in the context of duals of nuclear spaces). 

Next as a consequence of our proof of the L\'{e}vy-It\^{o} decomposition, we prove a L\'{e}vy-Khintchine formula for the characteristic function of a $\Phi'_{\beta}$-valued L\'{e}vy process (see Theorem \ref{levyKhintchineFormulaLevyProcessTheorem}). 
Moreover, by using the one-to-one correspondence between L\'{e}vy processes and infinitely divisible measures, in Section \ref{sectionLKTIDM} we prove the L\'{e}vy-Khintchine formula for the characteristic function of an infinitely divisible measure on $\Phi'_{\beta}$ (see Theorem \ref{levyKhintchineFormula}). More specifically, we prove that the characteristic function $\widehat{\mu}$ of an infinitely divisible measure $\mu$ on $\Phi'_{\beta}$ is of the form:
$$ \widehat{\mu}(\phi)=\exp \left[ i \goth{m}[\phi] - \frac{1}{2} \mathcal{Q}(\phi)^{2} + \int_{\Phi'_{\beta}} \left( e^{i f[\phi]} -1 - i f[\phi] \ind{ B_{\rho'}(1)}{f} \right) \nu(d f) \right], \quad \forall \phi \in \Phi,$$
where $\goth{m} \in \Phi'_{\beta}$, $\mathcal{Q}$ is a continuous Hilbertian semi-norm on $\Phi$, and $\nu$ is a L\'{e}vy measure on $\Phi'_{\beta}$ with corresponding Hilbertian norm $\rho'$. Here it is important to remark that our L\'{e}vy-Khintchine formula works in a case that is not covered by the formula proved by Dettweiler in \cite{Dettweiler:1976}. This is because our dual space is not assumed to be a complete Badrikian space as in \cite{Dettweiler:1976}. Finally, in Section \ref{sectionExampCommen}  we consider concrete examples of nuclear spaces, give some remarks, and compare our results with those on the literature.

The results obtained in this paper have extensive further applications. In particular, we have applied the  L\'{e}vy-It\^{o} decomposition to introduce stochastic integrals and to study stochastic evolution equations driven by L\'{e}vy noise in the dual of nuclear space (see \cite{FonsecaMora:2018SPDE}) and to characterize the semimartingales with deterministic characteristics (see \cite{FonsecaMora:Semi}), and we have applied the L\'{e}vy-Khintchine formula to the study of weak convergence of a sequence of L\'{e}vy processes (see \cite{FonsecaMora:Skorokhod}).  

\section{Preliminaries} \label{sectionPrelim}

\subsection{Nuclear Spaces And Their Strong Duals} \label{subsectionNuclSpace}

In this section we introduce our notation and review some of the key concepts on nuclear spaces and their dual spaces that we will need throughout this paper. For more information see \cite{Schaefer, Treves}.  

Let $\Phi$ be a locally convex space (over $\R$ or $\C$). If each bounded and closed subset of $\Phi$ is complete, then $\Phi$ is said to be \emph{quasi-complete}. On the other hand, the space $\Phi$ is called a \emph{barrelled} space if every  convex, balanced, absorbing and closed subset of $\Phi$ (i.e. a barrel) is a neighborhood of zero. 

If $p$ is a continuous semi-norm on $\Phi$ and $r>0$, the closed ball of radius $r$ of $p$ given by $B_{p}(r) = \left\{ \phi \in \Phi: p(\phi) \leq r \right\}$ is a closed, convex, balanced neighborhood of zero in $\Phi$. A continuous semi-norm (respectively a norm) $p$ on $\Phi$ is called \emph{Hilbertian} if $p(\phi)^{2}=Q(\phi,\phi)$, for all $\phi \in \Phi$, where $Q$ is a symmetric, non-negative bilinear form (respectively inner product) on $\Phi \times \Phi$. Let $\Phi_{p}$ be the Hilbert space that corresponds to the completion of the pre-Hilbert space $(\Phi / \mbox{ker}(p), \tilde{p})$, where $\tilde{p}(\phi+\mbox{ker}(p))=p(\phi)$ for each $\phi \in \Phi$. The quotient map $\Phi \rightarrow \Phi / \mbox{ker}(p)$ has a unique continuous linear extension that we denote by  $i_{p}:\Phi \rightarrow \Phi_{p}$.   

Let $q$ be another continuous Hilbertian semi-norm on $\Phi$ for which $p \leq q$. In this case, $\mbox{ker}(q) \subseteq \mbox{ker}(p)$. Moreover, the inclusion map from $\Phi / \mbox{ker}(q)$ into $\Phi / \mbox{ker}(p)$ is linear and continuous and it has a unique continuous linear extension that we denote by $i_{p,q}:\Phi_{q} \rightarrow \Phi_{p}$. Furthermore we have the following relation: $i_{p}=i_{p,q} \circ i_{q}$. 

We denote by $\Phi'$ the topological dual of $\Phi$ and by $f[\phi]$ the canonical pairing of elements $f \in \Phi'$, $\phi \in \Phi$. We denote by $\Phi'_{\beta}$ the dual space $\Phi'$ equipped with its \emph{strong topology} $\beta$, i.e. $\beta$ is the topology on $\Phi'$ generated by the family of semi-norms $\{ \eta_{B} \}$, where for each $B \subseteq \Phi$ bounded we have $\eta_{B}(f)=\sup \{ \abs{f[\phi]}: \phi \in B \}$ for all $f \in \Phi'$.  If $p$ is a continuous Hilbertian semi-norm on $\Phi$, then we denote by $\Phi'_{p}$ the Hilbert space dual to $\Phi_{p}$. The dual norm $p'$ on $\Phi'_{p}$ is given by $p'(f)=\sup \{ \abs{f[\phi]}:  \phi \in B_{p}(1) \}$ for all $ f \in \Phi'_{p}$. Moreover, the dual operator $i_{p}'$ corresponds to the canonical inclusion from $\Phi'_{p}$ into $\Phi'_{\beta}$ and it is linear and continuous. 

Let $p$ and $q$ be continuous Hilbertian semi-norms on $\Phi$ such that $p \leq q$.
The space of continuous linear operators (respectively Hilbert-Schmidt operators) from $\Phi_{q}$ into $\Phi_{p}$ is denoted by $\mathcal{L}(\Phi_{q},\Phi_{p})$ (respectively $\mathcal{L}_{2}(\Phi_{q},\Phi_{p})$) and the operator norm (respectively Hilbert-Schmidt norm) is denoted by $\norm{\cdot}_{\mathcal{L}(\Phi_{q},\Phi_{p})}$ (respectively $\norm{\cdot}_{\mathcal{L}_{2}(\Phi_{q},\Phi_{p})}$). We employ an analogous notation for operators between the dual spaces $\Phi'_{p}$ and $\Phi'_{q}$. 
  
Let us recall that a (Hausdorff) locally convex space $(\Phi,\mathcal{T})$ is called \emph{nuclear} if its topology $\mathcal{T}$ is generated by a family $\Pi$ of Hilbertian semi-norms such that for each $p \in \Pi$ there exists $q \in \Pi$, satisfying $p \leq q$ and the canonical inclusion $i_{p,q}: \Phi_{q} \rightarrow \Phi_{p}$ is Hilbert-Schmidt. Other equivalent definitions of nuclear spaces can be found in \cite{Pietsch, Treves}. Examples of nuclear spaces are given in Section \ref{sectionExampCommen}. 

Let $\Phi$ be a nuclear space. If $p$ is a continuous Hilbertian semi-norm  on $\Phi$, then the Hilbert space $\Phi_{p}$ is separable (see \cite{Pietsch}, Proposition 4.4.9 and Theorem 4.4.10, p.82). Now, let $\{ p_{n} \}_{n \in \N}$ be an increasing sequence of continuous Hilbertian semi-norms on $(\Phi,\mathcal{T})$. We denote by $\theta$ the locally convex topology on $\Phi$ generated by the family $\{ p_{n} \}_{n \in \N}$. The topology $\theta$ is weaker than $\mathcal{T}$. We  will call $\theta$ a weaker countably Hilbertian topology on $\Phi$ and we denote by $\Phi_{\theta}$ the space $(\Phi,\theta)$. The space $\Phi_{\theta}$ is a separable pseudo-metrizable (not necessarily Hausdorff) locally convex space and its dual space satisfies $\Phi'_{\theta}=\bigcup_{n \in \N} \Phi'_{p_{n}}$ (see \cite{FonsecaMora:2018}, Proposition 2.4). We denote the completion of $\Phi_{\theta}$ by  $\widetilde{\Phi_{\theta}}$ and its strong dual by $(\widetilde{\Phi_{\theta}})'_{\beta}$.

\subsection{Cylindrical and Stochastic Processes} \label{subSectionCylAndStocProcess}

Unless otherwise specified, in this section $\Phi$ will always denote a nuclear space over $\R$.
 
Let $\ProbSpace$ be a complete probability space. We denote by $L^{0} \ProbSpace$ the space of equivalence classes of real-valued random variables defined on $\ProbSpace$. We always consider the space $L^{0} \ProbSpace$ equipped with the topology of convergence in probability and in this case it is a complete, metrizable, topological vector space. 

For two Borel measures $\mu$ and $\nu$ on $\Phi'_{\beta}$, we denote by $\mu \ast \nu$ their \emph{convolution}. Recall that $\mu \ast \nu (A) = \int_{\Phi' \times \Phi'} \ind{A}{x+y} \mu(dx) \nu(dy)$, for any $A \in \mathcal{B}(\Phi'_{\beta})$. Denote $\nu^{\ast n}=\nu \ast \dots \ast \nu$ ($n$-times) and we use the convention $\nu^{0}=\delta_{0}$, where  $\delta_{f}$ denotes the Dirac measure on $\Phi'_{\beta}$ for $f \in \Phi'$. The restriction of a Borel measure $\nu$ on a set $A \in \mathcal{B}(\Phi'_{\beta})$ will be denoted by $\restr{\nu}{A}$. 

A Borel measure $\mu$ on $\Phi'_{\beta}$ is called a \emph{Radon measure} if for every $\Gamma \in \mathcal{B}(\Phi'_{\beta})$ and $\epsilon >0$, there exist a compact set $K_{\epsilon} \subseteq \Gamma$ such that $\mu(\Gamma \backslash K_{\epsilon}) < \epsilon$. In general not every Borel measure on $\Phi$ is Radon.
We denote by $\goth{M}_{R}^{b}(\Phi'_{\beta})$ and by $\goth{M}_{R}^{1}(\Phi'_{\beta})$ the spaces of all bounded Radon measures and of all Radon probability measures on $\Phi'_{\beta}$ respectively.  A subset $M \subseteq \goth{M}_{R}^{b}(\Phi'_{\beta})$ is called \emph{uniformly tight} if \begin{inparaenum}[(i)] 
\item $\sup \{ \mu(\Phi'_{\beta}) \st \mu \in M \} < \infty$, and \item for every $\epsilon >0$ there exist a compact $K \subseteq \Phi'_{\beta}$ such that $\mu (K^{c})< \epsilon$ for all $\mu \in M$. 
\end{inparaenum}
Also, a subset $M \subseteq \goth{M}_{R}^{b}(\Phi'_{\beta})$ is called \emph{shift tight} if for every $\mu \in M$ there exists $f_{\mu} \in \Phi'_{\beta}$ such that $\{ \mu \ast \delta_{f_{\mu}} : \mu \in M \}$ is uniformly tight. 

For any $n \in \N$ and any $\phi_{1}, \dots, \phi_{n} \in \Phi$, we define a linear map $\pi_{\phi_{1}, \dots, \phi_{n}}: \Phi' \rightarrow \R^{n}$ by
\begin{equation} \label{defiMapProjectionCylinder}
\pi_{\phi_{1}, \dots, \phi_{n}}(f)=(f[\phi_{1}], \dots, f[\phi_{n}]), \quad \forall \, f \in \Phi'. 
\end{equation}
The map $\pi_{\phi_{1}, \dots, \phi_{n}}$ is clearly linear and continuous. Let $M \subseteq \Phi$. A subset of $\Phi'$ of the form 
\begin{equation} \label{defiCylindricalSet}
\mathcal{Z}\left(\phi_{1}, \dots, \phi_{n}; A \right) = \left\{ f \in \Phi'\st \left(f[\phi_{1}], \dots, f[\phi_{n}]\right) \in A \right\}= \pi_{\phi_{1}, \dots, \phi_{n}}^{-1}(A)
\end{equation}
where $n \in \N$, $\phi_{1}, \dots, \phi_{n} \in M$ and $A \in \mathcal{B}\left(\R^{n}\right)$ is called a \emph{cylindrical set} based on $M$. The set of all the cylindrical sets based on $M$ is denoted by $\mathcal{Z}(\Phi',M)$. It is an algebra but if $M$ is a finite set then it is a $\sigma$-algebra. The $\sigma$-algebra generated by $\mathcal{Z}(\Phi',M)$ is denoted by $\mathcal{C}(\Phi',M)$ and it is called the \emph{cylindrical $\sigma$-algebra} with respect to $(\Phi',M)$. 
If $M=\Phi$, we write $\mathcal{Z}(\Phi')=\mathcal{Z}(\Phi',\Phi)$ and $\mathcal{C}(\Phi')=\mathcal{C}(\Phi',\Phi)$\index{.c CPhi@$\mathcal{C}(\Phi')$}. One can easily see from \eqref{defiCylindricalSet} that $\mathcal{Z}(\Phi') \subseteq \mathcal{B}(\Phi'_{\beta})$. Therefore, $\mathcal{C}(\Phi') \subseteq \mathcal{B}(\Phi'_{\beta})$. 
A function $\mu: \mathcal{Z}(\Phi') \rightarrow [0,\infty]$ is called a \emph{cylindrical measure} on $\Phi'$, if for each finite subset $M \subseteq \Phi'$ the restriction of $\mu$ to $\mathcal{C}(\Phi',M)$ is a measure. A cylindrical measure $\mu$ is said to be \emph{finite} if $\mu(\Phi')< \infty$ and a \emph{cylindrical probability measure} if $\mu(\Phi')=1$. The complex-valued function $\widehat{\mu}: \Phi \rightarrow \C$ defined by 
$$ \widehat{\mu}(\phi)= \int_{\Phi'} e^{i f[\phi]} \mu(df) = \int_{-\infty}^{\infty} e^{iz} \mu_{\phi}(dz), \quad \forall \, \phi \in \Phi, $$
where for each $\phi \in \Phi$, $\mu_{\phi} \defeq \mu \circ \pi_{\phi}^{-1}$, is called the \emph{characteristic function}\index{characteristic function} of $\mu$. In general a cylindrical measure on $\Phi'$ does not extend to a Borel measure on $\Phi'_{\beta}$. However, necessary and sufficient conditions for this can be given in terms of the continuity of its characteristic function by means of the Minlos theorem (see \cite{DaleckyFomin}, Theorem III.1.3, p.88).

A \emph{cylindrical random variable}\index{cylindrical random variable} in $\Phi'$ is a linear map $X: \Phi \rightarrow L^{0} \ProbSpace$. If $Z=\mathcal{Z}\left(\phi_{1}, \dots, \phi_{n}; A \right)$ is a cylindrical set, where $\phi_{1}, \dots, \phi_{n} \in \Phi$ and $A \in \mathcal{B}\left(\R^{n}\right)$, let 
\begin{equation*} 
\mu_{X}(Z) \defeq \Prob \left( ( X(\phi_{1}), \dots, X(\phi_{n})) \in A  \right) = \Prob \circ X^{-1} \circ \pi_{\phi_{1},\dots, \phi_{n}}^{-1}(A). 
\end{equation*}
The map $\mu_{X}$ is a cylindrical probability measure on $\Phi'$ and it is called the \emph{cylindrical distribution} of $X$. Conversely, to every cylindrical probability measure $\mu$ on $\Phi'$ there exists a canonical cylindrical random variable for which $\mu$ is its cylindrical distribution (see \cite{SchwartzRM}, p.256-8).  

If $X$ is a cylindrical random variable in $\Phi'$, the \emph{characteristic function} of $X$ is defined as the characteristic function $\widehat{\mu}_{X}: \Phi \rightarrow \C$ of its cylindrical distribution $\mu_{X}$. Hence $ \widehat{\mu}_{X}(\phi)= \Exp e^{i X(\phi)} $, $\forall \, \phi \in \Phi$. Also, we say that $X$ is \emph{$n$-integrable} if $ \Exp \left( \abs{X(\phi)}^{n} \right)< \infty$, $\forall \, \phi \in \Phi$, and has \emph{mean-zero} if $ \Exp \left( X(\phi) \right)=0$, $\forall \phi \in \Phi$. 

Let $X$ be a $\Phi'_{\beta}$-valued random variable, i.e. $X:\Omega \rightarrow \Phi'_{\beta}$ is a $\mathscr{F}/\mathcal{B}(\Phi'_{\beta})$-measurable map. We denote by $\mu_{X}$ the distribution of $X$, i.e. $\mu_{X}(\Gamma)=\Prob \left( X \in  \Gamma \right)$, $\forall \, \Gamma \in \mathcal{B}(\Phi'_{\beta})$, and it is a Borel probability measure on $\Phi'_{\beta}$. For each $\phi \in \Phi$ we denote by $X[\phi]$ the real-valued random variable defined by $X[\phi](\omega) \defeq X(\omega)[\phi]$, for all $\omega \in \Omega$. Then the mapping $\phi \mapsto X[\phi]$ defines a cylindrical random variable. Therefore the above concepts of characteristic function and integrability can be analogously defined for a $\Phi'_{\beta}$-valued random variable in terms of the cylindrical random variable induced/determined by $X$.  

If $X$ is a cylindrical random variable in $\Phi'$, a $\Phi'_{\beta}$-valued random variable $Y$ is a called a \emph{version} of $X$ if for every $\phi \in \Phi$, $X(\phi)=Y[\phi]$ $\Prob$-a.e. 
 
A $\Phi'_{\beta}$-valued random variable $X$ is called \emph{regular} if there exists a weaker countably Hilbertian topology $\theta$ on $\Phi$ such that $\Prob( \omega: X(\omega) \in \Phi'_{\theta})=1$. The following results establish alternative characterizations for regular random variables. 

\begin{theo}[\cite{FonsecaMora:2018}, Theorem 2.9]\label{theoCharacterizationRegularRV}
Let $X$ be a $\Phi'_{\beta}$-valued random variable. Consider the statements:
\begin{enumerate}
\item $X$ is regular.
\item The map $X: \Phi \rightarrow L^{0} \ProbSpace$, $\phi \mapsto X[\phi]$ is continuous. 
\item The distribution $\mu_{X}$ of $X$ is a Radon probability measure. 
\end{enumerate} 
Then $(1) \Leftrightarrow (2)$  and $(2) \Rightarrow (3)$. Moreover, if $\Phi$ is barrelled, we have $(3) \Rightarrow (1)$.  
\end{theo}

Let $J=[0,\infty)$ or $J=[0,T]$ for some $T>0$. We say that $X=\{ X_{t} \}_{t \in J}$ is a \emph{cylindrical process} in $\Phi'$ if $X_{t}$ is a cylindrical random variable, for each $t \in J$. Clearly any $\Phi'_{\beta}$-valued stochastic processes $X=\{ X_{t} \}_{t \in J}$ defines a cylindrical process under the prescription: $X[\phi]=\{ X_{t}[\phi] \}_{t \in J}$, for each $\phi \in \Phi$. We will say that it is the \emph{cylindrical process determined/induced} by $X$.

A $\Phi'_{\beta}$-valued processes $Y=\{Y_{t}\}_{t \in J}$ is said to be a $\Phi'_{\beta}$-valued \emph{version} of the cylindrical process $X=\{X_{t}\}_{t \in J}$ on $\Phi'$ if for each $t \in J$, $Y_{t}$ is a $\Phi'_{\beta}$-valued version of $X_{t}$.  

Let $X=\{ X_{t} \}_{t \in J}$ be a $\Phi'_{\beta}$-valued process.
We say that $X$ is \emph{continuous} (respectively \emph{c\`{a}dl\`{a}g}) if for $\Prob$-a.e. $\omega \in \Omega$, the \emph{sample paths} $t \mapsto X_{t}(w) \in \Phi'_{\beta}$ of $X$ are continuous (respectively right-continuous with left limits). On the other hand, we say that $X$ is \emph{regular} if for every $t \in J$, $X_{t}$ is a regular random variable. The following two results contain some useful properties of $\Phi'_{\beta}$- valued regular processes. For proofs see Chapter 1 in \cite{FonsecaMoraThesis}. 

\begin{prop}\label{propCondiIndistingProcess} Let $X=\left\{ X_{t} \right\}_{t \in J}$ and $Y=\left\{ Y_{t} \right\}_{t \in J}$ be $\Phi'_{\beta}$- valued regular stochastic processes such that for each $\phi \in \Phi$, $X[\phi]=\left\{ X_{t}[\phi] \right\}_{t \in J}$ is a version of $Y[\phi]=\left\{ Y_{t}[\phi] \right\}_{t \in J}$. Then $X$ is a version of $Y$. Furthermore, if $X$ and $Y$ are right-continuous then they are indistinguishable. 
\end{prop}

\begin{prop} \label{independenceStochProcessDualSpace} Let $X^{1}=\left\{ X^{1}_{t} \right\}_{t \in J}$, $\dots$, $X^{k}=\left\{ X^{k}_{t} \right\}_{t \in J}$ be $\Phi'_{\beta}$- valued regular processes. Then $X^{1}, \dots, X^{k}$ are independent if and only if for all $n \in \N$ and $\phi_{1}, \dots, \phi_{n} \in \Phi$, the $\R^{n}$-valued processes 
$\{(X^{j}_{t}[\phi_{1}],\dots, X^{j}_{t}[\phi_{n}]): t \in J \}$, $j=1,\dots,k$, are independent.
\end{prop}

The following result known as the \emph{regularization theorem} plays a fundamental role in this work. It establishes conditions for a cylindrical stochastic process in $\Phi'$ to have a regular continuous or c\`{a}dl\`{a}g version.  

\begin{theo}[Regularization Theorem; \cite{FonsecaMora:2018}, Theorem 3.2]\label{theoRegularizationTheoremCadlagContinuousVersion}
Let $X=\{X_{t} \}_{t \geq 0}$ be a cylindrical process in $\Phi'$ satisfying:
\begin{enumerate}
\item For each $\phi \in \Phi$, the real-valued process $X(\phi)=\{ X_{t}(\phi) \}_{t \geq 0}$ has a continuous (respectively c\`{a}dl\`{a}g) version.
\item For every $T > 0$, the family $\{ X_{t}: t \in [0,T] \}$ of linear maps from $\Phi$ into $L^{0} \ProbSpace$ is equicontinuous.  
\end{enumerate}
Then there exists a countably Hilbertian topology $\vartheta_{X}$ on $\Phi$ 
and a $(\widetilde{\Phi_{\vartheta_{X}}})'_{\beta}$-valued continuous (respectively c\`{a}dl\`{a}g) process $Y= \{ Y_{t} \}_{t \geq 0}$, such that for every $\phi \in \Phi$, $Y[\phi]= \{ Y_{t}[\phi] \}_{t \geq 0}$ is a version of $X(\phi)= \{ X_{t}(\phi) \}_{t \geq 0}$. Moreover $Y$ is a $\Phi'_{\beta}$-valued, regular, continuous (respectively c\`{a}dl\`{a}g) version of $X$ that is unique up to indistinguishable versions. 
\end{theo}

The following result is a particular case of the regularization theorem that establishes conditions for the existence of a regular continuous or c\`{a}dl\`{a}g version with finite moments and taking values in one of the Hilbert spaces $\Phi'_{q}$. 

\begin{theo}[\cite{FonsecaMora:2018}, Theorem 4.3] \label{theoExistenceCadlagContVersionHilbertSpaceUniformBoundedMoments}
Let $X=\{X_{t} \}_{t \geq 0}$ be a cylindrical process in $\Phi'$ satisfying:
\begin{enumerate}
\item For each $\phi \in \Phi$, the real-valued process $X(\phi)=\{ X_{t}(\phi) \}_{t \geq 0}$ has a continuous (respectively c\`{a}dl\`{a}g) version.
\item There exists $n \in \N$ and a continuous Hilbertian semi-norm $\varrho$ on $\Phi$ such that for all $T>0$ there exists $C(T)>0$ such that 
\begin{equation} \label{uniformBoundMomentsByHilbertSeminorm}
\Exp \left( \sup_{t \in [0,T]} \abs{X_{t}(\phi)}^{n} \right) \leq C(T) \varrho(\phi)^{n}, \quad \forall \, \phi \in \Phi.
\end{equation} 
\end{enumerate}
Then there exists a continuous Hilbertian semi-norm $q$ on $\Phi$, $\varrho \leq q$, such that $i_{\varrho,q}$ is Hilbert-Schmidt and there exists a $\Phi'_{q}$-valued continuous (respectively c\`{a}dl\`{a}g) process $Y=\{ Y_{t} \}_{t \geq 0}$, satisfying:
\begin{enumerate}[label=(\alph*)]
\item For every $\phi \in \Phi$, $Y[\phi]= \{ Y_{t}[\phi] \}_{t \geq 0}$ is a version of $X(\phi)= \{ X_{t}(\phi) \}_{t \geq 0}$, 
\item For every $T>0$, $\Exp \left( \sup_{t \in [0,T]} q'(Y_{t})^{n} \right) < \infty$.   
\end{enumerate} 
Furthermore $Y$ is a $\Phi'_{\beta}$-valued continuous (respectively c\`{a}dl\`{a}g) version of $X$ that is unique up to indistinguishable versions.
\end{theo}

\section{L\'{e}vy Processes and Infinitely Divisible Measures.}\label{sectionLPBP}

In this section we study the relationship between L\'{e}vy processes and infinitely divisible measures. The link between these two concepts are the cylindrical L\'{e}vy processes and the semigroups of probability measures. 

\subsection{Infinitely Divisible Measures and Convolution Semigroups in the Strong Dual.} \label{subSectionInfDivMeas}

Let $\Psi$ be a locally convex space. A measure $\mu \in \goth{M}_{R}^{1}(\Psi'_{\beta})$ is called \emph{infinitely divisible} if for every $n\in \N$ there exist a \emph{$n$-th root} of $\mu$, i.e. a measure $\mu_{n} \in \goth{M}_{R}^{1}(\Psi'_{\beta})$ such that $\mu=\mu_{n}^{\ast n}$. We denote by $\goth{I}(\Psi'_{\beta})$ the set of all infinitely divisible measures on $\Psi'_{\beta}$.  

A family $\{ \mu_{t} \}_{t \geq 0} \subseteq \goth{M}_{R}^{1}(\Psi'_{\beta})$ is said to be a \emph{convolution semigroup} if $\mu_{s} \ast \mu_{t} = \mu_{s+t}$ for any $s,t \geq 0$ and $ \mu_{0}= \delta_{0}$. Moreover, we say that the convolution semigroup is \emph{continuous} if the mapping $t \mapsto \mu_{t}$ from $[0,\infty)$ into $\goth{M}_{R}^{1}(\Psi'_{\beta})$ is continuous in the weak topology. 

The following result follows easily from the definition of continuous convolution semigroup. 

\begin{prop} \label{propConvoSemiIsInfinDivisible}
If $\{ \mu_{t} \}_{t \geq 0}$ is a convolution semigroup in $\goth{M}_{R}^{1}(\Psi'_{\beta})$, then $\forall \, t \geq 0$ $\mu_{t} \in \goth{I}(\Psi'_{\beta})$.
\end{prop}

Now, to prove the converse of Proposition \ref{propConvoSemiIsInfinDivisible} we will need the following definitions. 

Let $\mu$ be an infinitely divisible measure on $\Psi'_{\beta}$. We define the \emph{root set} of $\mu$ by 
$$ R(\mu) \defeq \bigcup_{n \geq 1} \left\{ \nu^{m}: \nu \in  \goth{M}_{R}^{1}(\Psi'_{\beta}) \mbox{ with } \nu^{n}= \mu, \, 1 \leq m \leq n \right\}.$$ 
We say that $\mu$ is \emph{root compact} if its root set $R(\mu)$ is uniformly tight.

We are ready for the main result of this section. As stated in its proof, the main arguments are based on several results due to E. Siebert (see \cite{Siebert:1974, Siebert:1976}).

\begin{theo} \label{theoCorrespInfDivisAndConvSemigroups}
Assume that $\Psi$ is a locally convex space for which $\Psi'_{\beta}$ is quasi-complete. If $\mu \in \goth{I}(\Psi'_{\beta})$, then there exists a unique continuous convolution semigroup $\{ \mu_{t} \}_{t \geq 0}$ in $\goth{M}_{R}^{1}(\Phi'_{\beta})$ such that $\mu_{1}=\mu$.  
\end{theo}
\begin{prf}
First as $\Psi'_{\beta}$ is locally convex and $\mu \in \goth{I}(\Psi'_{\beta})$, there exists a rational continuous convolution semigroup $\{ \nu_{t} \}_{t \in \Q \cap [0,\infty) }$ in $\goth{M}_{R}^{1}(\Psi'_{\beta})$ such that $\nu_{1}=\mu$ (see \cite{Siebert:1974}, Korollar 5.4). 

Now, as $\mu=\nu_{1}=\nu_{1/q}^{\ast q}$, then $\nu_{1/q}$ is a root of $\mu$ for each $q \in \N \setminus \{ 0\}$. But as for $p, q \in \N \setminus \{ 0\}$ we have $\nu^{\ast p/q}=\nu_{1/q}^{\ast p}$, then we have that $\nu_{t} \in R(\mu)$ for each $t \in \Q \cap [0,1]$. 

On the other hand, as $\mu$ is tight (is Radon) and $\Psi'_{\beta}$ is a  quasi-complete locally convex space, the root set $R(\mu)$ of $\mu$ is uniformly tight (see \cite{Siebert:1974}, Satz 6.2 and 6.4). Hence the set $\{ \nu_{t} \}_{t \in \Q \cap [0,1] }$ is uniformly tight and by Prokhorov's theorem it is relatively compact. This last property guarantees the existence of a (unique) continuous convolution semigroup $\{ \mu_{t} \}_{t \geq 0}$ in $\goth{M}_{R}^{1}(\Phi'_{\beta})$ such that $\nu_{t}=\mu_{t}$ for each $t \in \Q \cap [0,\infty)$ (see \cite{Siebert:1976}, Proposition 5.3). Therefore, $\mu=\nu_{1}=\mu_{1}$.   
\end{prf}

The following result will be of great importance in further developments. 
 
\begin{lemm} \label{lemmLocalUnifTightnessConvSemigroup}
Assume that $\Psi'_{\beta}$ is quasi-complete and let $\{ \mu_{t} \}_{t \geq 0}$ be a continuous convolution semigroup in $\goth{M}_{R}^{1}(\Phi'_{\beta})$. Then, $\forall \, T>0$ $\{ \mu_{t}: t \in [0,T]\}$ is uniformly tight. 
\end{lemm}
\begin{prf}
Let $T>0$. Using similar arguments to those used in the proof of Theorem \ref{theoCorrespInfDivisAndConvSemigroups}. we can show that $\{ \mu_{t} \}_{t \in \Q \cap [0,T] } \subseteq  R(\mu_{T})$, and because $\mu_{T}$ is tight, the root set $R(\mu_{T})$ of $\mu_{T}$ is uniformly tight (see \cite{Siebert:1974}, Satz 6.2 and 6.4), and hence $\{ \mu_{t} \}_{t \in \Q \cap [0,T] }$ is also uniformly tight. 

Now note that for each $r \in \I \cup [0,T]$ the continuity of the semigroup $\{ \mu_{t} \}_{t \geq 0}$ shows that $\mu_{r}= \lim_{q \searrow r, q \in \Q \cap [0,T]} \mu_{q}$ in the weak topology. Therefore $\{ \mu_{t} \}_{t \in [0,T] }$ is in the weak closure $\overline{\{ \mu_{t} \}_{t \in \Q \cap [0,T] }}$ of $\{ \mu_{t} \}_{t \in \Q \cap [0,T] }$. But because the weak closure of an uniformly tight family in $\goth{M}_{R}^{1}(\Phi'_{\beta})$ is also uniformly tight (see \cite{VakhaniaTarieladzeChobanyan}, Theorem I.3.5), then it follows that $\overline{\{ \mu_{t} \}_{t \in \Q \cap [0,T] }}$ is uniformly tight and hence $\{ \mu_{t} \}_{t \in [0,T] }$ is uniformly tight too. 
\end{prf}

\subsection{L\'{e}vy Processes and Cylindrical L\'{e}vy Processes} \label{subSectionLPCLP}

From now on and unless otherwise specified, $\Phi$ will always be a nuclear space over $\R$.

We start with our definition of L\'{e}vy processes on the dual of a nuclear space. 

\begin{defi} \label{defiLevyProcess}
A $\Phi'_{\beta}$-valued process $L=\left\{ L_{t} \right\}_{t\geq 0}$ is called a \emph{L\'{e}vy process} if it satisfies:
\begin{enumerate}
\item $L_{0}=0$ a.s.
\item $L$ has \emph{independent increments}, i.e. for any $n \in \N$, $0 \leq t_{1}< t_{2} < \dots < t_{n} < \infty$ the $\Phi'_{\beta}$-valued random variables $L_{t_{1}},L_{t_{2}}-L_{t_{1}}, \dots, L_{t_{n}}-L_{t_{n-1}}$ are independent.  
\item L has \emph{stationary increments}, i.e. for any $0 \leq s \leq t$, $L_{t}-L_{s}$ and $L_{t-s}$ are identically distributed. 
\item For every $t \geq 0$ the distribution $\mu_{t}$ of $L_{t}$ is a Radon measure and the mapping $t \mapsto \mu_{t}$ from $[0, \infty)$ into $\goth{M}_{R}^{1}(\Phi'_{\beta})$ is continuous at $0$ in the weak topology.
\end{enumerate}
\end{defi}

The probability distributions of a $\Phi'_{\beta}$-valued L\'{e}vy process satisfy the following properties:

\begin{theo} \label{theoLevyDefinesConvSemig}
If $L=\left\{ L_{t} \right\}_{t\geq 0}$ is a L\'{e}vy process in $\Phi'_{\beta}$, the family of probability distributions $\{ \mu_{L_{t}} \}_{t\geq 0}$ of $L$ is a continuous convolution semigroup in $\goth{M}_{R}^{1}(\Phi'_{\beta})$. Moreover each $\mu_{L_{t}}$ is infinitely divisible for every $t \geq 0$. Furthermore if $\Phi'_{\beta}$ is quasi-complete, then for each $T>0$ the family $\{ \mu_{L_{t}}: t \in [0,T]\}$ is uniformly tight.
\end{theo} 
\begin{prf}
The semigroup property of $\{ \mu_{L_{t}} \}_{t\geq 0}$ is an easy consequence of the stationary and independent increments properties of $L$. The weak continuity is part of our definition of L\'{e}vy process. The fact that each $\mu_{L_{t}}$ is infinitely divisible follows from Proposition \ref{propConvoSemiIsInfinDivisible}. If $\Phi'_{\beta}$ is quasi-complete the uniform tightness of $\{ \mu_{L_{t}}: t \in [0,T]\}$ follows from Lemma \ref{lemmLocalUnifTightnessConvSemigroup}.  
\end{prf}

Following the definition given in Applebaum and Riedle \cite{ApplebaumRiedle:2010} for cylindrical L\'{e}vy processes in Banach spaces, we introduce the following definition. 

\begin{defi} A cylindrical process $L=\{ L_{t} \}_{t \geq 0}$ in $\Phi'$ is said to be a \emph{cylindrical L\'{e}vy process} if  $\forall \, n\in \N$, $\phi_{1}, \dots, \phi_{n} \in \Phi$, the $\R^{n}$-valued process $\{ (L_{t}(\phi_{1}), \dots, L_{t}(\phi_{n})) \}_{t \geq 0}$ is a L\'{e}vy process.  
\end{defi}

\begin{lemm} \label{lemmLevyProcessIsCylindricalLevy}
Every $\Phi'_{\beta}$-valued L\'{e}vy process $L=\{ L_{t} \}_{t \geq 0}$ determines a cylindrical L\'{e}vy process in $\Phi'$.    
\end{lemm}
\begin{prf}
Let $n\in \N$ and $\phi_{1}, \dots, \phi_{n} \in \Phi$. It is clear that $(L_{0}[\phi_{1}], \dots, L_{0}[\phi_{n}])=0$ $\Prob$-a.e. The fact that 
$\{ (L_{t}[\phi_{1}], \dots, L_{t}[\phi_{n}]) \}_{t \geq 0}$ has stationary and independent increments follows from the corresponding properties of $L$ as a $\Phi'_{\beta}$-valued process (see Proposition \ref{independenceStochProcessDualSpace}). Finally the stochastic continuity of $\{ (L_{t}[\phi_{1}], \dots, L_{t}[\phi_{n}]) \}_{t \geq 0}$ is a consequence of the weak continuity of the map $t \mapsto \mu_{t}$ (see \cite{ApplebaumLPSC}, Proposition 1.4.1). 
\end{prf}

The following result gives sufficient conditions for a cylindrical L\'{e}vy process to have a version that is a L\'{e}vy process in $\Phi'_{\beta}$. This result will be of great importance in the proof of the one-to-one correspondence of L\'{e}vy processes and infinitely divisible measures in Section \ref{subsectionCLPIDM}.   
		
\begin{theo} \label{theoCylindrLevyProcessHaveLevyCadlagVersion}
Let $L=\{ L_{t} \}_{t \geq 0}$ be a cylindrical L\'{e}vy process in $\Phi'$ such that for every $T > 0$, the family $\{ L_{t}: t \in [0,T] \}$ of linear maps from $\Phi$ into $L^{0} \ProbSpace$ is equicontinuous.  
Then there exists a countably Hilbertian topology $\vartheta_{L}$ on $\Phi$ and a $(\widetilde{\Phi_{\vartheta_{L}}})'_{\beta}$-valued c\`{a}dl\`{a}g process $Y= \{ Y_{t} \}_{t \geq 0}$, such that for every $\phi \in \Phi$, $Y[\phi]= \{ Y_{t}[\phi] \}_{t \geq 0}$ is a version of $L(\phi)= \{ L_{t}(\phi) \}_{t \geq 0}$. Moreover $Y$ is a $\Phi'_{\beta}$-valued, regular, c\`{a}dl\`{a}g L\'{e}vy process that is a version of $L$ and that is unique up to indistinguishable versions. 
\end{theo}
\begin{prf} 
First as for each $\phi \in \Phi$ the real-valued process $L(\phi)=\{ L_{t}(\phi) \}_{t \geq 0}$ is a L\'{e}vy process, then it has a c\`{a}dl\`{a}g version (see Theorem 2.1.8 of Applebaum \cite{ApplebaumLPSC}, p.87). Hence $L$ satisfies all the conditions of the regularization theorem (Theorem \ref{theoRegularizationTheoremCadlagContinuousVersion}) and this theorem shows the existence of a countably Hilbertian topology $\vartheta_{L}$ on $\Phi$ and a $(\widetilde{\Phi_{\vartheta_{L}}})'_{\beta}$-valued c\`{a}dl\`{a}g process $Y= \{ Y_{t} \}_{t \geq 0}$, such that for every $\phi \in \Phi$, $Y[\phi]= \{ Y_{t}[\phi] \}_{t \geq 0}$ is a version of $L(\phi)= \{ L_{t}(\phi) \}_{t \geq 0}$. Moreover, it is a consequence of the regularization theorem that $Y$ is a $\Phi'_{\beta}$-valued, regular, c\`{a}dl\`{a}g version of $L$ that is unique up to indistinguishable versions.

Our next step is to show that $Y$ is a $\Phi'_{\beta}$-valued L\'{e}vy process. First, as $Y_{0}[\phi]=L_{0}(\phi)=0$ $\Prob$-a.e. for every $\phi \in \Phi$, it follows that $Y_{0}=0$ $\Prob$-a.e. (Proposition \ref{propCondiIndistingProcess}). Second, as for each $\phi_{1}, \dots, \phi_{n} \in \Phi$, the $\R^{n}$-valued process  
$\{ (L_{t}(\phi_{1}), \dots, L_{t}(\phi_{n})) \}_{t \geq 0}$ has independent and stationary increments, and because for each $t \geq 0$ we have that  
$$ (L_{t}(\phi_{1}), \dots, L_{t}(\phi_{n}))= (Y_{t}[\phi_{1}], \dots, Y_{t}[\phi_{n}]), \quad \Prob- \mbox{a.e.}, $$
then the $\R^{n}$-valued process  
$\{ (Y_{t}[\phi_{1}], \dots, Y_{t}[\phi_{n}]) \}_{t \geq 0}$ also has independent and stationary increments for every $\phi_{1}, \dots, \phi_{n} \in \Phi$. Hence because $Y$ is a $\Phi'_{\beta}$-valued regular process, it then follows from Propositions \ref{propCondiIndistingProcess} and  \ref{independenceStochProcessDualSpace} that $Y$ has independent and stationary increments.

Now the fact that $Y$ is a $\Phi'_{\beta}$-valued regular process and Theorem \ref{theoCharacterizationRegularRV}  show that for each $t \geq 0$ the probability distribution $\mu_{t}$ of $Y_{t}$ is a Radon measure. 

Our final step to show that $Y$ is a $\Phi'_{\beta}$-valued L\'{e}vy process is to prove that the mapping $t \mapsto \mu_{t}$ from $[0, \infty)$ into $\goth{M}_{R}^{1}(\Phi'_{\beta})$ is continuous in the weak topology. 

Let $t \geq 0$. Our objective is to show that for any net $\{ s_{\alpha} \}$ in $[0,\infty)$ such that $\lim_{\alpha} s_{\alpha}=t$ we have $\lim_{\alpha} \mu_{s_{\alpha}}=\mu_{t}$ in the weak topology on $\goth{M}_{R}^{1}(\Phi'_{\beta})$. As convergence of filter bases is only determined by terminal sets, we can choose without loss of generality some sufficiently large $T>0$ and consider only nets in $[0,T]$ satisfying $\{ s_{\alpha} \}$ such that $\lim_{\alpha} s_{\alpha}=t$. Let $\{ s_{\alpha} \}$ be such a net. 

First as for each $\phi \in \Phi$, $Y[\phi]= \{ Y_{t}[\phi] \}_{t \geq 0}$ is stochastically continuous, it follows that the family $\{ Y_{s_{\alpha}}[\phi] \}$ converges in probability to $Y_{t}[\phi]$. Now this last property in turn shows that $\lim_{\alpha} \widehat{\mu}_{s_{\alpha}}(\phi)=\widehat{\mu}_{t}(\phi)$ for every $\phi \in \Phi$. 

Now for each $r \geq 0$ denote by $\nu_{r}$ the cylindrical distribution of the cylindrical random variable $L_{r}$. Then the equicontinuity of the family $\{ L_{r}: r \in [0,T] \}$ of linear maps from $\Phi$ into $L^{0} \ProbSpace$ implies that the family of characteristic functions $\{ \widehat{\nu}_{r} \}_{r \in [0,T]}$ is equicontinuous at zero. But as for each $r \geq 0$, $\widehat{\nu}_{r}(\phi) = \widehat{\mu}_{r}(\phi)$ for all $\phi \in \Phi$, we then have that the family of characteristic functions $\{ \widehat{\mu}_{r} \}_{r \in [0,T]}$ of $\{ Y_{r} \}_{r \in [0,T]}$ is equicontinuous at zero. However, as $\Phi$ is a nuclear space the equicontinuity of $\{ \widehat{\mu}_{r} \}_{r \in [0,T]}$ at zero implies that $\{ \mu_{r} \}_{r \in [0,T]}$ is uniformly tight (see \cite{DaleckyFomin}, Lemma III.2.3, p.103-4). This last in turn shows that $\{ \mu_{s_{\alpha}} \}$ is uniformly tight, and the Prokhorov's theorem (see \cite{DaleckyFomin}, Theorem III.2.1, p.98) implies that the family $\{ \mu_{s_{\alpha}} \}$ is relatively compact in the weak topology. Because we also have that $\lim_{\alpha} \widehat{\mu}_{s_{\alpha}}(\phi)=\widehat{\mu}_{t}(\phi)$ for every $\phi \in \Phi$, we then conclude that $\lim_{\alpha} \mu_{s_{\alpha}}=\mu_{t}$ in the weak topology (see \cite{VakhaniaTarieladzeChobanyan}, Theorem IV.3.1, p.224-5). Consequently, the map $t \mapsto \mu_{t}$ is continuous in the weak topology and $Y$ is a $\Phi'_{\beta}$-valued L\'{e}vy process.   
\end{prf}

An important variation of the above theorem is the following:

\begin{theo} \label{theoCylindrLevyHilbContHaveLevyCadlagVersHilbSpace}
Let $L=\{ L_{t} \}_{t \geq 0}$ be a cylindrical L\'{e}vy process in $\Phi'$. Assume that there exist $n \in \N$ and a continuous Hilbertian semi-norm $\varrho$ on $\Phi$ such that for all $T>0$ there is a $C(T)>0$ such that 
\begin{equation*} 
\Exp \left( \sup_{t \in [0,T]} \abs{L_{t}(\phi)}^{n} \right) \leq C(T) \varrho(\phi)^{n}, \quad \forall \, \phi \in \Phi.
\end{equation*} 
Then there exists a continuous Hilbertian semi-norm $q$ on $\Phi$, $\varrho \leq q$, such that $i_{\varrho,q}$ is Hilbert-Schmidt and there exists a $\Phi'_{q}$-valued c\`{a}dl\`{a}g L\'{e}vy process $Y=\{ Y_{t}\}_{t \geq 0}$, satisfying:
\begin{enumerate}[label=(\alph*)]
\item For every $\phi \in \Phi$, $Y[\phi]= \{ Y_{t}[\phi] \}_{t \geq 0}$ is a version of $L(\phi)= \{ L_{t}(\phi) \}_{t \geq 0}$, 
\item For every $T>0$, $\Exp \left( \sup_{t \in [0,T]} q'(Y_{t})^{n} \right) < \infty$.   
\end{enumerate} 
Moreover $Y$ is a $\Phi'_{\beta}$-valued, regular, c\`{a}dl\`{a}g version of $L$ that is unique up to indistinguishable versions. Furthermore if the real-valued process $L(\phi)$ is continuous for each $\phi \in \Phi$, then $Y$ can be chosen to be continuous in $\Phi'_{q}$ and hence in $\Phi'_{\beta}$. 
\end{theo}
\begin{prf}
The existence of the $\Phi'_{q}$-valued c\`{a}dl\`{a}g process $Y=\{ Y_{t}\}_{t \geq 0}$ satisfying the conditions in the statement of the theorem follows from Theorem \ref{theoExistenceCadlagContVersionHilbertSpaceUniformBoundedMoments}. Finally similar arguments to those used in the proof of Theorem \ref{theoCylindrLevyProcessHaveLevyCadlagVersion} show that $Y$ is a $\Phi'_{q}$-valued L\'{e}vy process. 
\end{prf}

We now provide a sufficient condition for the existence of a c\`{a}dl\`{a}g version for a $\Phi'_{\beta}$-valued L\'{e}vy process. 

\begin{theo} \label{theoConditCadlagVersionLevyProc}
Let $L=\{ L_{t} \}_{t \geq 0}$ be a $\Phi'_{\beta}$-valued L\'{e}vy process. Suppose that for every $T > 0$, the family $\{ L_{t}: t \in [0,T] \}$ of linear maps from $\Phi$ into $L^{0} \ProbSpace$ given by $\phi \mapsto L_{t}[\phi]$ is equicontinuous. 
Then $L$ has a $\Phi'_{\beta}$-valued, regular, c\`{a}dl\`{a}g version $\tilde{L}=\{ \tilde{L}_{t} \}_{t \geq 0}$ that is also a L\'{e}vy process. Moreover there exists a countably Hilbertian topology $\vartheta_{L}$ on $\Phi$ such that $\tilde{L}$ is a $(\widetilde{\Phi_{\vartheta_{L}}})'_{\beta}$-valued c\`{a}dl\`{a}g process.  
\end{theo}
\begin{prf}
First note that our assumption on $L$ implies that $L$ is regular. This is because for each $t \geq 0$ the fact that $L_{t}: \Phi \rightarrow L^{0} \ProbSpace$ is continuous shows that $L_{t}$ is a regular random variable in $\Phi'_{\beta}$ (Theorem \ref{theoCharacterizationRegularRV}).  

Now as $L$ is a $\Phi'_{\beta}$-valued L\'{e}vy process the cylindrical process determined by $L$ is a cylindrical L\'{e}vy process (Lemma \ref{lemmLevyProcessIsCylindricalLevy}). But from our assumptions on $L$, this cylindrical L\'{e}vy process satisfies the assumptions in Theorem \ref{theoCylindrLevyProcessHaveLevyCadlagVersion}. Therefore,
there exists a $\Phi'_{\beta}$-valued, regular, c\`{a}dl\`{a}g L\'{e}vy process $\tilde{L}= \{ \tilde{L}_{t} \}_{t \geq 0}$, such that for every $\phi \in \Phi$, $\tilde{L}_{t}[\phi]= L_{t}[\phi]$ $\Prob$-a.e. for each $t \geq 0$. This last property together with the fact that both $\tilde{L}$ and $L$ are regular process shows that $\tilde{L}$ is a version of $L$ (Proposition \ref{propCondiIndistingProcess}). Finally from Theorem \ref{theoCylindrLevyProcessHaveLevyCadlagVersion} there exists a countably Hilbertian topology $\vartheta_{L}$ on $\Phi$ such that $\tilde{L}$ is a $(\widetilde{\Phi_{\vartheta_{L}}})'_{\beta}$-valued c\`{a}dl\`{a}g process. 
\end{prf}

\begin{coro} \label{coroExistCadlagVersionLevyInBarrelledNuclearSpace}
If $\Phi$ is a barrelled nuclear space and $L=\{ L_{t} \}_{t \geq 0}$ is a $\Phi'_{\beta}$-valued L\'{e}vy process, then $L$ has a $\Phi'_{\beta}$-valued c\`{a}dl\`{a}g version satisfying the properties given in Theorem \ref{theoConditCadlagVersionLevyProc}.  
\end{coro}
\begin{prf} First as $\Phi$ is barrelled then $\Phi'_{\beta}$ is quasi-complete (see \cite{Schaefer}, Theorem IV.6.1, p.148). Then given $T>0$, it follows from Theorem \ref{theoLevyDefinesConvSemig} that the family $\{ \mu_{L_{t}}: t \in [0,T]\}$ of its probability distributions is uniformly tight. But because $\Phi$ is barrelled, the family of their characteristic functions $\{ \widehat{\mu}_{L_{t}}: t \in [0,T]\}$ is equicontinuous (see \cite{BogachevMT}, Vol. II, Corollary 7.13.10, p.126), and consequently the family of linear maps $\{ L_{t}: t \in [0,T] \}$ is also equicontinuous (see \cite{VakhaniaTarieladzeChobanyan}, Proposition IV.3.4, p.231). Therefore, $L$ satisfies the assumptions on Theorem \ref{theoConditCadlagVersionLevyProc} and the result follows. 
\end{prf}

Finally the next result provides sufficient conditions for the existence of a c\`{a}dl\`{a}g version that is a L\'{e}vy process with finite $n$-th moment in some of the Hilbert spaces $\Phi'_{q}$. This result will play an important role in our proof of the L\'{e}vy-It\^{o} decomposition in Section \ref{subsectionLID}. 

\begin{theo} \label{theoCondLevyCadlagVersHilbSpace}
Let $L=\{ L_{t} \}_{t \geq 0}$ be a $\Phi'_{\beta}$-valued L\'{e}vy process. 
Assume that there exist $n \in \N$ and a continuous Hilbertian semi-norm $\varrho$ on $\Phi$ such that for all $T>0$ there is a $C(T)>0$ such that 
\begin{equation*} 
\Exp \left( \sup_{t \in [0,T]} \abs{L_{t}[\phi]}^{n} \right) \leq C(T) \varrho(\phi)^{n}, \quad \forall \, \phi \in \Phi.
\end{equation*}
Then there exists a continuous Hilbertian semi-norm $q$ on $\Phi$, $\varrho \leq q$, such that $i_{\varrho,q}$ is Hilbert-Schmidt and a $\Phi'_{q}$-valued, c\`{a}dl\`{a}g (continuous if $L$ is continuous), L\'{e}vy process $\tilde{L}=\{ \tilde{L}_{t} \}_{t \geq 0}$ that is a version of $L$. Moreover $\Exp \left( \sup_{t \in [0,T]} q'(\tilde{L}_{t})^{n} \right) < \infty$ $\forall \, T>0$.   
\end{theo}
\begin{prf}
The proof follows from Theorem \ref{theoCylindrLevyHilbContHaveLevyCadlagVersHilbSpace} and similar arguments to those used in the proof of Theorem  \ref{theoConditCadlagVersionLevyProc}. 
\end{prf}

\subsection{Correspondence of L\'{e}vy Processes and Infinitely Divisible Measures}\label{subsectionCLPIDM} 

We have  shown in Theorem \ref{theoLevyDefinesConvSemig} that to  every $\Phi'_{\beta}$-valued L\'{e}vy process $L=\{ L_{t} \}_{t \geq 0}$ the probability distribution $\mu_{L_{t}}$ of $L_{t}$ is infinitely divisible for each $t \geq 0$. In this section we will show that if the space $\Phi$ is barrelled and nuclear (see examples in Section \ref{subsectionNuclSpace}), to every infinitely divisible measure $\mu$ on $\Phi'_{\beta}$ there corresponds a $\Phi'_{\beta}$-valued L\'{e}vy process $L$ such that $\mu_{L_{1}}=\mu$. 

In order to prove our main result (Theorem \ref{theoInfiDivisMeasuImpliLevyProc}), we will need the following theorem that establishes the existence of a cylindrical L\'{e}vy process from a given family of cylindrical probability measures with some semigroup properties. 
We formulate our result in the more general context of Hausdorff locally convex spaces. The definitions of cylindrical probability measure and cylindrical L\'{e}vy process are exactly the same to those given in Sections \ref{subSectionCylAndStocProcess} and \ref{subSectionLPCLP}. 

\begin{theo}\label{theoCylSemiGroupImpliesCylLevyProc}
Let $\Psi$ be a Hausdorff locally convex space. Let $\{ \mu_{t} \}_{t \geq 0}$ be a family of cylindrical measures on $\Psi'$ such that for every finite collection $\psi_{1}, \psi_{2}, \dots, \psi_{n} \in \Psi$, the family $\{ \mu_{t} \circ \pi_{\psi_{1}, \psi_{2}, \dots, \psi_{n}}^{-1} \}_{t \geq 0}$ is a continuous convolution semigroup of probability measures on $\R^{n}$. Then, there exists a cylindrical process $L=\{ L_{t} \}_{t \geq 0}$ in $\Psi'$ defined on a probability space $\ProbSpace$, such that:
\begin{enumerate}
\item For every $t \geq 0$, $\psi_{1}, \psi_{2}, \dots, \psi_{n} \in \Psi$ and $\Gamma \in \mathcal{B}(\R^{n})$,  
$$\Prob \left( (L_{t}(\psi_{1}), L_{t}(\psi_{2}), \dots, L_{t}(\psi_{n})) \in \Gamma \right) = \mu_{t} \circ \pi_{\psi_{1}, \psi_{2}, \dots, \psi_{n}}^{-1}(\Gamma). $$ 
\item $L$ is a cylindrical L\'{e}vy process in $\Psi'$. 
\end{enumerate}  
\end{theo}

For the proof of Theorem \ref{theoCylSemiGroupImpliesCylLevyProc} we will need to deal with projective systems of measure spaces. For the convenience of the reader we recall their definition (see \cite{RaoSPGT} p.17-19 for more details).  	

Let $\{ (\Omega_{\alpha},\Sigma_{\alpha},P_{\alpha}): \alpha \in D\}$ be a family of measure spaces, where $D$ is a directed set, and let $\{ g_{\alpha,\beta}:  \alpha < \beta, \alpha, \beta \in D\}$ be a family of mappings such that: \begin{inparaenum}[(i)] \item $g_{\alpha \beta}: \Omega_{\beta} \rightarrow \Omega_{\alpha}$, and $g_{\alpha \beta}^{-1}(\Sigma_{\alpha}) \subseteq \Sigma_{\beta}$, \item for any $\alpha < \beta < \gamma$, $g_{\alpha \gamma}= g_{\alpha \beta} \circ g_{\beta \gamma}$ and $g_{\alpha \alpha}=$identity, and \item for every $\alpha < \beta$, $P_{\alpha}=P_{\beta} \circ g_{\alpha \beta}^{-1} $.  \end{inparaenum} Then the  collection $\{ (\Omega_{\alpha},\Sigma_{\alpha},P_{\alpha},  g_{\alpha \beta} )_{\alpha < \beta}: \alpha, \beta \in D\}$ is called a projective system of measure spaces (of Hausdorff topological spaces if each $\Omega_{\alpha}$ is a Hausdorff topological space, $\Sigma_{\alpha}$ is its Baire $\sigma$-algebra, the measure $P_{\alpha}$ is regular in the measure theory sense and each $g_{\alpha \beta}$ is continuous).

\begin{proof}[Proof of Theorem \ref{theoCylSemiGroupImpliesCylLevyProc}] 
Our first objective is to define a projective system of Hausdorff topological spaces for which the probability space $\ProbSpace$ will be its projective limit (see \cite{RaoSPGT}). 

Let $\mathbb{F}$ be the set of all finite ordered collections of elements of $\Psi$. For any $F=(\psi_{1}, \psi_{2}, \dots, \psi_{n}) \in \mathbb{F}$, define $\pi_{F} \defeq \pi_{\psi_{1}, \psi_{2}, \dots, \psi_{n}}$, where recall that $\pi_{\psi_{1}, \psi_{2}, \dots, \psi_{n}}(f)=(f[\psi_{1}],f[\psi_{2}], \dots, f[\psi_{n}])$, for all $f \in \Psi'$. Then, it is clear that the map $\pi_{F}: \Psi' \mapsto \Omega^{F}$ is continuous, where $\Omega^{F} \defeq \R^{n}$. 

Now for $F \in \mathbb{F}$, define $\mu_{t}^{F} \defeq \mu_{t} \circ \pi_{F}^{-1}$, for all $t \geq 0$. Then, from our assumptions on $\{ \mu_{t} \}_{t \geq 0}$ we have that $\{ \mu_{t}^{F} \}_{t \geq 0}$ is a continuous convolution semigroup of probability measures on $\Omega^{F}$.     

Consider on $\mathbb{F}$ the partial order $ \leq_{\mathbb{F}}$ determined by the set inclusion. For any $F, G \in \mathbb{F}$ satisfying $F \leq_{\mathbb{F}} G$, denote by $g_{F,G}: \Omega^{G} \rightarrow \Omega^{F}$ the canonical projection from $\Omega^{G}$ into $\Omega^{F}$. For any $F, G, H \in \mathbb{F}$ satisfying $F \leq_{\mathbb{F}} G \leq_{\mathbb{F}} H$, it follows from the definitions above that we have:
\begin{equation} \label{projectionMapsBBF}
g_{F,H}=g_{F,G} \circ g_{G,H}, \quad g_{F,F} = \mbox{ identity on } \Omega^{F}, 
\end{equation}      
\begin{equation} \label{compatiMeasuAndProjectFG}
\mu_{t}^{F}=\mu_{t}^{G} \circ g_{F,G}^{-1}. 
\end{equation}
Now, let $\mathbb{A}=\{ \{ (t_{i},\psi_{i})\}_{i=1}^{n}: \, n \in \N, \, 0 \leq t_{1} \leq t_{2} \leq \dots \leq t_{n}, \, \psi_{1}, \psi_{2}, \dots, \psi_{n} \in \Psi \}$. Then, $(\mathbb{A}, \leq_{\mathbb{A}} )$ is a directed set when $\leq_{\mathbb{A}} $ is the partial order on $\mathbb{A}$ defined as follows:
\newline for $A= \{ (t_{i}, \psi_{i})\}_{i=1}^{n}$, $B= \{ (s_{j},\phi_{j})\}_{j=1}^{m} \in \mathbb{A}$, we say $A \leq_{\mathbb{A}} B$ if 
$$\{ t_{1}, t_{2}, \dots, t_{n} \} \subseteq \{ s_{1}, s_{2}, \dots, s_{m} \} \quad \mbox{and} \quad F=(\psi_{1}, \psi_{2}, \dots, \psi_{n}) \leq_{\mathbb{F}} G=(\phi_{1}, \phi_{2}, \dots, \phi_{m}). $$

For $A, B \in \mathbb{A}$ as above, define $\Omega^{A}= \Omega^{F}_{t_{1}} \times \Omega^{F}_{t_{2}} \times \cdots \times \Omega^{F}_{t_{n}}$ and $\Omega^{B}= \Omega^{G}_{s_{1}} \times \Omega^{G}_{s_{2}} \times \cdots \times \Omega^{G}_{s_{m}}$, where $\Omega^{F}_{t_{i}} \defeq \Omega^{F}$ for  $i=1, \dots, n$, and $\Omega^{G}_{s_{i}} \defeq \Omega^{G}$ for $j=1, \dots, m$. Clearly, $\Omega^{A}$ and $\Omega^{B}$ are Hausdorff topological vector spaces. 

Now note that if $A \leq_{\mathbb{A}} B$, then from the definition of $\leq_{\mathbb{A}} $ we have $ \{ t_{i} \}_{i=1}^{n} \subseteq \{ s_{j} \}_{j=1}^{m}$. Let $s_{j_{1}}, \dots, s_{j_{n}}$ given by $s_{j_{i}}=t_{i}$, for $i=1, \dots, n$. Define the projection $g_{A,B}: \Omega^{B} \rightarrow \Omega^{A}$ by the prescription:
\begin{align}
(w_{s_{1}}, w_{s_{2}}, \dots, w_{s_{n}}) \in \Omega^{B} 
& \mapsto  (g_{F,G}(w_{s_{j_{1}}}), g_{F,G}(w_{s_{j_{2}}}), \dots, g_{F,G}(w_{s_{j_{n}}})) \nonumber \\
&   = (g_{F,G}(w_{t_{1}}), g_{F,G}(w_{t_{2}}), \dots, g_{F,G}(w_{t_{n}})) \in \Omega^{A}. \label{projectionOmegaBtoOmegaA}  
\end{align}
If $C=\{ (r_{k}, \varphi_{k})\}_{k=1}^{p} \in \mathbb{A}$ is such that $A \leq_{\mathbb{A}} B \leq_{\mathbb{A}} C$, and if we take $H= (\varphi_{1}, \varphi_{2}, \dots, \varphi_{p} ) \in \mathbb{F}$, then it is clear from \eqref{projectionMapsBBF} that:
\begin{equation} \label{projectionMapsBBA}
g_{A,C}=g_{A,B} \circ g_{B,C}, \quad g_{A,A} = \mbox{ identity on } \Omega^{A}, 
\end{equation}
Now, for $A \in \mathbb{A}$ as above, define $\mu_{A}$ by 
\begin{equation} \label{defiFamilyMeasuresProjectiveSystem}
\mu_{A}(\Gamma_{1} \times \cdots \times \Gamma_{n})= \int_{\Gamma_{1}} \mu^{F}_{t_{1}}(dw_{1}) \int_{\Gamma_{2}} \mu^{F}_{t_{2}- t_{1}}(dw_{2} - w_{1}) \dots \int_{\Gamma_{n}} \mu^{F}_{t_{n}- t_{n-1}}(dw_{n} - w_{n-1}), 
\end{equation}
for $\Gamma_{i} \in \mathcal{B}(\Omega^{F}_{t_{i}})$, $\forall i=1, \dots, n$. Then $\mu_{A}$ can be extended to a unique measure on $\Omega^{A}$. 

Now let $\Gamma_{i} \in \mathcal{B}(\Omega^{F}_{t_{i}})$, $\forall i=1, \dots, n$. From \eqref{projectionOmegaBtoOmegaA} it follows that for $A \leq_{\mathbb{A}} B$ we have: 
\begin{equation} \label{imageSetsOnProjectionMapsAB}
g_{A,B}^{-1}(\Gamma_{1} \times \cdots \times \Gamma_{n})= \Sigma_{1} \times \cdots \times \Sigma_{m}, \quad \mbox{where} \quad  
\Sigma_{j} =
\begin{cases}
\Omega^{G}_{s_{j}}, & \mbox{if } s_{j} \notin \{ s_{j_{1}}, \dots, s_{j_{n}} \}, \\
g_{F,G}^{-1}(\Gamma_{j}), & \mbox{if } s_{j} \in \{ s_{j_{1}}, \dots, s_{j_{n}} \}.  
\end{cases}
\end{equation}
Hence from \eqref{compatiMeasuAndProjectFG}, \eqref{defiFamilyMeasuresProjectiveSystem} and \eqref{imageSetsOnProjectionMapsAB}, it follows that:
\begin{flalign}
& \mu_{B}(g_{A,B}^{-1}(\Gamma_{1} \times \cdots \times \Gamma_{n})) \nonumber \\
& = \mu_{B}(\Sigma_{1} \times \cdots \times \Sigma_{m}) \nonumber \\
& = \int_{g_{F,G}^{-1}(\Gamma_{1})} \mu^{G}_{s_{j_{1}}}(dw_{1}) \int_{g_{F,G}^{-1}(\Gamma_{2})} \mu^{G}_{s_{j_{2}}- s_{j_{1}}}(dw_{2} - w_{1}) \dots \int_{g_{F,G}^{-1}(\Gamma_{n})} \mu^{G}_{s_{j_{n}}- s_{j_{n-1}}}(dw_{n} - w_{n-1}) \nonumber \\
& = \int_{\Gamma_{1}} \mu^{G}_{t_{1}} \circ g_{F,G}^{-1}(dw_{1}) \int_{\Gamma_{2}} \mu^{G}_{t_{2}- t_{1}} \circ g_{F,G}^{-1} (dw_{2} - w_{1}) \dots \int_{\Gamma_{n}} \mu^{G}_{t_{n}- t_{n-1}} \circ g_{F,G}^{-1}(dw_{n} - w_{n-1}) \nonumber \\
& = \int_{\Gamma_{1}} \mu^{F}_{t_{1}}(dw_{1}) \int_{\Gamma_{2}} \mu^{F}_{t_{2}- t_{1}}(dw_{2} - w_{1}) \dots \int_{\Gamma_{n}} \mu^{F}_{t_{n}- t_{n-1}}(dw_{n} - w_{n-1}) \nonumber \\
& = \mu_{A}(\Gamma_{1} \times \cdots \times \Gamma_{n}). \label{compatiMeasuAndProjecAB}
\end{flalign}
where on the passage from the first to the second line we used that $\{ \mu^{G}_{t} \}_{t \geq 0}$ is a convolution semigroup of probability measures on $\Omega^{G}$. Then from a standard argument it follows that \eqref{compatiMeasuAndProjecAB} extends to
\begin{equation} \label{projectSystemCompatBBA}
\mu_{B} \circ g_{A,B}^{-1}= \mu_{A}, \quad \forall \, A, B \in \mathbb{A}, \, A \leq_{\mathbb{A}} B. 
\end{equation}
We then conclude that $\{ (\Omega^{A}, \mathcal{B}(\Omega^{A}), \mu_{A}, g_{A,B})_{A \leq_{\mathbb{A}} B}: \, A, B \in \mathbb{A} \}$ is a projective system of Hausdorff topological vector spaces. Hence from a generalization of the Kolmogorov's Extension Theorem (see \cite{RaoSPGT}, Theorem 1.3.4, p.20-1), the latter system admits a unique limit $\ProbSpace$ where $\Omega = \R^{\mathbb{A}} \cong \varprojlim (\Omega^{A}, g_{A,B})$, $\mathcal{F}= \sigma  \left( \bigcup_{A \in \mathbb{A}} g_{A}^{-1} (\mathcal{B}(\Omega^{A})) \right)$ and $\Prob = \varprojlim \mu_{A}$, where $g_{A}: \Omega \rightarrow \Omega^{A}$ is the canonical projection determined by the projections $g_{A,B}$.  

On the above, $\varprojlim (\Omega^{A}, g_{A,B})$ is the subset of $\times_{A \in \mathbb{A}} \Omega^{A}$ of all the elements $(\omega_{A})_{A \in \mathbb{A}}$ such that for $A \leq_{\mathbb{A}} B$ we have $g_{A,B}(\omega_{B})=\omega_{A}$. On the other hand, $g_{A}$ is the projection $(\omega_{A})_{A \in \mathbb{A}} \mapsto \omega_{A} \in \Omega^{A}$. Also $\Prob = \varprojlim \mu_{A}$ means that $\Prob$ is a (probability) measure on $\Omega$ that satisfies
\begin{equation} \label{compPropMuAAndProbMeasu} 
\mu_{A}= \Prob \circ g_{A}^{-1}, \quad \forall \, A \in \mathbb{A}. 
\end{equation}

Our next step is to define a cylindrical process $L=\{ L_{t} \}_{t \geq 0}$ in $\Psi'$ defined on the probability space $\ProbSpace$ that satisfies the conditions (1) and (2) on the statement of the theorem. 

First it is clear that $\Omega$ can be embedded in $\R^{\R_{+} \times \Psi} = \times_{(t,\psi)} \R^{(t,\psi)}$, where $\R^{(t,\psi)}=\R$ for each $(t,\psi) \in \R_{+} \times \Psi$. This is an easy consequence of the fact that $\mathbb{A}$ consist of finite collections of elements of $\R_{+} \times \Psi$. 

Now let $\tilde{I}: \R^{\R_{+} \times \Psi} \rightarrow \R^{\R_{+} \times \Psi}$ be the identity mapping. Define $\tilde{L}: \R_{+} \times \Psi \rightarrow L^{0} \ProbSpace$ by $\tilde{L}(t,\psi)=g_{(t,\psi)} \circ \tilde{I}$, where $g_{(t,\psi)}: \Omega \rightarrow \R$ is the coordinate projection. Then it follows from the definition of $\tilde{L}$ that:
$$ \tilde{L}(t,\psi)(\omega)=g_{(t,\psi)} (\tilde{I}(\omega))= g_{(t,\psi)} (\omega) = \omega ((t,\psi)) \in \R, \quad \forall \, \omega \in \R^{\R_{+} \times \Psi}. $$    
We clearly have that for each $(t,\psi) \in \R^{\R_{+} \times \Psi}$, $\tilde{L}(t,\psi)$ is a real-valued random variable since $\{ \omega : g_{(t,\psi)} (\tilde{I}(\omega)) < a\} \subseteq \Omega$ is a cylinder set in $\mathcal{F}$. Moreover for $A= \{ (t_{i}, \psi_{i})\}_{i=1}^{n} \in \mathbb{A}$, we have that $\tilde{L} \circ A$ given by $\tilde{L} \circ A (\omega)  \defeq ( \tilde{L}(t_{1},\psi_{1})(\omega),  \dots, \tilde{L}(t_{n},\psi_{n})(\omega))$ is a random vector because 
\begin{equation}
( \tilde{L}(t_{1},\psi_{1})(\omega),  \dots, \tilde{L}(t_{n},\psi_{n})(\omega)) = (g_{(t_{1},\psi_{1})} \circ \tilde{I}(\omega),	 \dots, g_{(t_{n},\psi_{n})} \circ \tilde{I}(\omega) ) = g_{A} \circ \tilde{I}(\omega) \label{defiRandomVectorTildeLOnA}
\end{equation}
and $\{ \omega: g_{A} \circ \tilde{I}(\omega) < a \} \subseteq \Omega$ is also a cylinder set in $\mathcal{F}$.  Then, from \eqref{compPropMuAAndProbMeasu} and \eqref{defiRandomVectorTildeLOnA} we have: 
\begin{equation} \label{compatMuAAndTildeL}
\mu_{A}= \Prob \circ g_{A}^{-1} = \Prob \circ (\tilde{L} \circ A)^{-1}, \quad \forall \, A \in \mathcal{A}. 
\end{equation}
Therefore $\mu_{A}$ is the distribution of the random vector $\tilde{L} \circ A$. Moreover for any $t \geq 0$ and $\psi_{1}, \psi_{2}, \dots, \psi_{n}$, it follows from our definition of $\mathbb{A}$ that $A=\{ (t,\psi_{i}) \}_{i=1}^{n} \in \mathbb{A}$, and from \eqref{defiFamilyMeasuresProjectiveSystem}, \eqref{defiRandomVectorTildeLOnA} and \eqref{compatMuAAndTildeL} we have for this $A$ that for every $\Gamma \in \mathcal{B}(\R^{n})$,
\begin{equation} \label{compMutAndTildeLOnCylinders}
\Prob \left( ( \tilde{L}(t,\psi_{1}), \tilde{L}(t,\psi_{2}), \dots, \tilde{L}(t,\psi_{n})) \in \Gamma \right) = \mu_{A} (\Gamma)= \mu_{t} \circ \pi^{-1}_{\psi_{1}, \psi_{2}, \dots, \psi_{n}} (\Gamma). 
\end{equation}  
Now fix $t \geq 0$. We will show the linearity of the map $\tilde{L}(t,\cdot): \Psi \rightarrow L^{0} \ProbSpace$.  

Let $\psi_{1}, \psi_{2} \in \Psi$. Consider  $F=(  \psi_{1},  \psi_{2},  \psi_{1}+ \psi_{2})$.
Then, recall that $\Omega^{F}=\R^{3}$ and in this case $\pi_{F}: \Psi' \rightarrow \R^{3}$ is given by $f \mapsto  (f[\psi_{1}], f[\psi_{2}],f[ \psi_{1}+  \psi_{2}])$.
If $\sigma: \R^{3} \rightarrow \R$ is given by $(a,b,c) \mapsto a+b-c$, then it is clear that $\sigma$ is continuous and that $\sigma \circ \pi_{F}=0$.
It then follows that for $\Gamma \in \mathcal{B}(\R)$, $\mu_{t} \circ \pi_{F}^{-1}( \sigma^{-1}(\Gamma))=0$ if $0 \notin \Gamma$ and $\mu_{t} \circ \pi_{F}^{-1}( \sigma^{-1}(\Gamma))=1$ if $0 \in \Gamma$. Hence, 
$\mu_{t} \circ \pi_{F}^{-1}$ is supported by the plane $\sigma^{-1}(\{ 0 \})=\{ (a,b,c): a+b-c=0\}$ of $\R^{3}$. But then,  we have from \eqref{compMutAndTildeLOnCylinders} that
\begin{equation*}
\Prob \left( (\tilde{L}(t,\psi_{1}), \tilde{L}(t,\psi_{2}),  \tilde{L}(t,\psi_{1}+\psi_{2}) ) \in \sigma^{-1}(\{ 0 \}) \right)
= \mu_{t} \circ \pi_{F}^{-1}( \sigma^{-1}(\{ 0\}))=1.
\end{equation*}
Therefore 
\begin{equation} \label{linearSumTildeL}
\tilde{L}(t,\psi_{1})+ \tilde{L}(t,\psi_{2})=  \tilde{L}(t,\psi_{1}+\psi_{2}) \quad \Prob-\mbox{a.e.}
\end{equation}
On the other hand, for any $\alpha \in \R$, $\psi \in \Psi$, if we consider $F=( \psi,\alpha \psi )$, then $\pi_{F}: \Psi' \rightarrow \R^{2}$ is given by $f \mapsto (f[\psi],f[\alpha \psi])$, and if we consider $\sigma: \R^{2} \rightarrow \R$ given by $(p,q) \mapsto \alpha p-q$, by using similar arguments to those used above we can show that 
\begin{equation} \label{linearScalarTildeL}
\alpha \tilde{L}(t,\psi)= \tilde{L}(t,\alpha \psi) \quad \Prob-\mbox{a.e.}
\end{equation}
Hence \eqref{linearSumTildeL} and \eqref{linearScalarTildeL} show that for a fixed $t \geq 0$ the map $\tilde{L}(t,\cdot): \Psi \rightarrow L^{0} \ProbSpace$ is linear.

Now define $L=\{ L_{t} \}_{t \geq 0}$, $L_{t}: \Psi \rightarrow L^{0} \ProbSpace$ by $L_{t}(\psi)(\omega)=\tilde{L}(t,\psi)(\omega)$, for all $t \geq 0$, $\psi \in \Psi$ and $\omega \in \Omega$. The linearity of the map $\tilde{L}(t,\cdot): \Psi \rightarrow L^{0} \ProbSpace$ for every $t \geq 0$ shows that $L=\{ L_{t} \}_{t \geq 0}$ is a cylindrical stochastic process in $\Psi'$. Moreover it follows from \eqref{compMutAndTildeLOnCylinders} that $\forall t \geq 0$, $\psi_{1}, \psi_{2}, \dots, \psi_{n} \in \Psi$, $\Gamma \in \mathcal{B}(\R^{n})$ we have
\begin{equation} \label{compMutAndCylProcLOnCylinders}
\Prob \left( ( L_{t}(\psi_{1}), L_{t}(\psi_{2}), \dots, L_{t}(\psi_{n})) \in \Gamma \right) = \mu_{t} \circ \pi^{-1}_{\psi_{1}, \psi_{2}, \dots, \psi_{n}} (\Gamma). 
\end{equation} 
Now we will show that $L=\{ L_{t} \}_{t \geq 0}$ is a cylindrical L\'{e}vy process in $\Psi$. Fix $\psi_{1}, \psi_{2}, \dots, \psi_{n} \in \Psi$. We have to show that  $\{ ( L_{t}(\psi_{1}), L_{t}(\psi_{2}), \dots, L_{t}(\psi_{n})) \}_{t \geq 0}$ is a $\R^{n}$-valued L\'{e}vy process.

First it follows from \eqref{defiFamilyMeasuresProjectiveSystem}, \eqref{compMutAndTildeLOnCylinders} and \eqref{compMutAndCylProcLOnCylinders} that for any $t_{1} < t_{2} < \dots < t_{n}$ and any bounded measurable function $f$ on $\R^{n^{2}}$, we have 
\begin{multline} \label{expOfCylLevyForBoundFunctF}
 \Exp \left[ f \left( ( L_{t_{1}}(\psi_{1}), L_{t_{1}}(\psi_{2}), \dots, L_{t_{1}}(\psi_{n})),  \dots, ( L_{t_{n}}(\psi_{1}), L_{t_{n}}(\psi_{2}), \dots, L_{t_{n}}(\psi_{n})) \right) \right] \\
 = \int \cdots \int f (w_{1}, w_{1}+w_{2}, \dots, w_{1}+w_{2}+ \dots +w_{n}) \\
 \times \mu^{F}_{t_{1}}(dw_{1}) \mu^{F}_{t_{2}- t_{1}}(dw_{2} - w_{1}) \dots  \mu^{F}_{t_{n}- t_{n-1}}(dw_{n} - w_{n-1}), 
\end{multline}
where $F=(\psi_{1}, \psi_{2}, \dots, \psi_{n})  \in \mathbb{F}$. 
Then by following similar arguments to those used on the proof of Theorem 2.7.10 in \cite{Sato} p.36, the independent and stationary increments of the $\R^{n}$-valued process $\{ ( L_{t}(\psi_{1}), L_{t}(\psi_{2}), \dots, L_{t}(\psi_{n})) \}_{t \geq 0}$ can be deduced by fixing $z_{1}, \dots, z_{n} \in \R^{n}$ and setting
$$f(w_{1}, w_{2}, \dots, w_{n}) = \exp \left( i \sum_{j=1}^{n} \inner{z_{j}}{w_{j}-w_{j-1}} \right), \quad \forall \,  w_{1}, w_{2}, \dots, w_{n} \in \R^{n}, \mbox{ with } w_{0}=0. $$

Finally the fact that the process $\{ ( L_{t}(\psi_{1}), L_{t}(\psi_{2}), \dots, L_{t}(\psi_{n})) \}_{t \geq 0}$ is stochastically continuous is a consequence of  \eqref{compMutAndCylProcLOnCylinders} and our assumption that $\{ \mu_{t} \circ \pi_{\psi_{1}, \psi_{2}, \dots, \psi_{n}}^{-1} \}_{t \geq 0}$ is a continuous convolution semigroup of probability measures on $\R^{n}$ (see \cite{ApplebaumLPSC}, Proposition 1.4.1). Thus, we have shown that $\{ ( L_{t}(\psi_{1}), L_{t}(\psi_{2}), \dots, L_{t}(\psi_{n})) \}_{t \geq 0}$ is a $\R^{n}$-valued L\'{e}vy process for any $\psi_{1}, \psi_{2}, \dots, \psi_{n} \in \Psi$ and consequently $L=\{ L_{t} \}_{t \geq 0}$ is a cylindrical L\'{e}vy process in $\Psi$. 
\end{proof}

We are ready for the main result of this section. 

\begin{theo} \label{theoInfiDivisMeasuImpliLevyProc}  
Let $\Phi$ be a barrelled nuclear space. If $\mu$ is an infinitely divisible measure on $\Phi'_{\beta}$, there exists a $\Phi'_{\beta}$-valued, regular, c\`{a}dl\`{a}g L\'{e}vy process $L=\left\{ L_{t} \right\}_{t\geq 0}$ such that $\mu_{L_{1}}=\mu$. 
\end{theo}

\begin{prf}   
First note that as $\Phi$ is barrelled then $\Phi'_{\beta}$ is quasi-complete (see \cite{Schaefer}, Theorem IV.6.1, p.148). Therefore it follows from Theorem \ref{theoCorrespInfDivisAndConvSemigroups} that there exists a unique continuous convolution semigroup $\{ \mu_{t} \}_{t \geq 0}$ in $\goth{M}_{R}^{1}(\Phi'_{\beta})$ such that $\mu_{1}=\mu$.  

Now it is clear that the cylindrical measures determined by the family $\{ \mu_{t} \}_{t \geq 0}$ satisfy that for every finite collection $\phi_{1}, \phi_{2}, \dots, \phi_{n} \in \Phi$, the family $\{ \mu_{t} \circ \pi^{-1}_{\phi_{1}, \phi_{2}, \dots, \phi_{n}} \}_{t \geq 0}$ is a continuous convolution semigroup of probability measures on $\R^{n}$. Then Theorem \ref{theoCylSemiGroupImpliesCylLevyProc} shows the existence of a cylindrical L\'{e}vy process $L=\{ L_{t} \}_{t \geq 0}$ in $\Phi'$ defined on a probability space $\ProbSpace$, such that for every $t \geq 0$, $\phi_{1}, \phi_{2}, \dots, \phi_{n} \in \Phi$ and $\Gamma \in \mathcal{B}(\R^{n})$,  
\begin{equation} \label{mutIsExtensionOfCylLevy}
\Prob \left( (L_{t}(\phi_{1}), L_{t}(\phi_{2}), \dots, L_{t}(\phi_{n})) \in \Gamma \right) = \mu_{t} \circ \pi_{\phi_{1}, \phi_{2}, \dots, \phi_{n}}^{-1}(\Gamma).
\end{equation}
Now given $T>0$, since $\Phi$ is barrelled, it follows from Lemma \ref{lemmLocalUnifTightnessConvSemigroup} that $\{ \mu_{t} : t \in [0,T]\}$ is uniformly tight. Then Theorem III.2.7 in \cite{DaleckyFomin} implies that the characteristic functions $\{ \widehat{\mu_{t}}: t \in [0,T]\}$ are equicontinuous.  
Thus as $L=\{ L_{t} \}_{t \geq 0}$ is a cylindrical L\'{e}vy process in $\Phi'$, it follows from 
Theorem \ref{theoCylindrLevyProcessHaveLevyCadlagVersion} that there exists a $\Phi'_{\beta}$-valued, regular, c\`{a}dl\`{a}g L\'{e}vy process $\tilde{L}=\{ \tilde{L}_{t} \}_{t \geq 0}$ that is a version of $L=\{ L_{t} \}_{t \geq 0}$. 
Moreover it follows from \eqref{mutIsExtensionOfCylLevy} that for every $t \geq 0$, $\phi_{1}, \phi_{2}, \dots, \phi_{n} \in \Phi$ and $\Gamma \in \mathcal{B}(\R^{n})$, 
\begin{align*}
\mu_{\tilde{L}_{t}} \circ \pi_{\phi_{1}, \phi_{2}, \dots, \phi_{n}}^{-1}(\Gamma)
& =\Prob \left( \tilde{L}_{t} \in \pi_{\phi_{1}, \phi_{2}, \dots, \phi_{n}}^{-1}(\Gamma) \right)\\
& =\Prob \left( (L_{t}(\phi_{1}), L_{t}(\phi_{2}), \dots, L_{t}(\phi_{n})) \in \Gamma \right)  = \mu_{t} \circ \pi_{\phi_{1}, \phi_{2}, \dots, \phi_{n}}^{-1}(\Gamma).
\end{align*}
Hence for every $t \geq 0$, the measures $\mu_{\tilde{L}_{t}}$ and $\mu_{t}$ coincide on all the cylindrical sets, but as both measures are Radon measures this is enough to conclude that that $\mu_{\tilde{L}_{t}}=\mu_{t}$. Now, as $\mu_{1}=\mu$, we then have that $\mu_{\tilde{L}_{1}}=\mu$. 
\end{prf}

\subsection{Wiener Processes in the Dual of a Nuclear Space}\label{subsectionWCPP}

In this section we quickly review some properties of Wiener processes in $\Phi'_{\beta}$ proved by K. It\^{o} \cite{Ito} and that we will need later for our proof of the L\'{e}vy-It\^{o} decomposition. 

\begin{defi} \label{wienerProcess} A $\Phi'_{\beta}$-valued continuous L\'{e}vy process $W=\left\{ W_{t} \right\}_{t\geq 0}$ is called a $\Phi'_{\beta}$-valued \emph{Wiener process}.
A $\Phi'_{\beta}$-valued process $G=\left\{ G_{t} \right\}_{t\geq 0}$ is called \emph{Gaussian} if for any $n \in \N$ and any $\phi_{1}, \dots, \phi_{n} \in \Phi$,  $\left\{ (G_{t}[\phi_{1}], \dots,G_{t}[\phi_{n}]) \st t \geq 0 \right\}$ is a Gaussian process in $\R^{n}$. 
\end{defi}

\begin{theo} [\cite{Ito}, Theorem 2.7.1]\label{propertiesWienerProcess}
Let $W=\left\{ W_{t} \right\}_{t\geq 0}$ be a $\Phi'_{\beta}$-valued Wiener process. Then, $W$ is Gaussian and hence square integrable. Moreover there exists $\goth{m} \in \Phi'$ and a continuous Hilbertian semi-norm $\mathcal{Q}$ on $\Phi$, called respectively the mean and the covariance functional of $W$, such that 
\begin{equation}
\Exp \left( W_{t} [\phi] \right) = t \goth{m} [\phi], \quad  \forall \, \phi \in \Phi, \, t \geq 0.  \label{meanWienerProcess}
\end{equation}   
\begin{equation}
\Exp \left( \left(W_{t}-t \goth{m} \right)[\phi] \left(W_{s}-s \goth{m} \right)[\varphi] \right) = ( t \wedge s ) \mathcal{Q} (\phi, \varphi), \quad \forall \, \phi, \varphi \in \Phi, \, s, t \geq 0,  \label{covarianceFunctWienerProcess}
\end{equation}
where in \eqref{covarianceFunctWienerProcess}, $\mathcal{Q}(\cdot,\cdot)$ corresponds to the continuous, symmetric, non-negative bilinear form on $\Phi \times \Phi$ associated to $\mathcal{Q}$. Furthermore the characteristic function of $W$ is given by  
\begin{equation} \label{charactFunctionWienerProcess}
\Exp \left( e^{i W_{t}[\phi]} \right)=\exp\left(i t \goth{m}[\phi] - \frac{t}{2} \mathcal{Q}(\phi)^{2} \right), \quad \mbox{ for each } t \geq 0, \,  \phi \in \Phi.
\end{equation}
\end{theo}

\begin{theo}[\cite{Ito}, Theorem 2.7.2] \label{existenceWienerProcess} 
Given $\goth{m} \in \Phi'$ and a continuous Hilbertian semi-norm $\mathcal{Q}$ on $\Phi$, there exists a $\Phi'_{\beta}$-valued Wiener process $W=\left\{ W_{t} \right\}_{t\geq 0}$ such that $\goth{m}$ and $\mathcal{Q}$ are the mean and covariance functional of $W$. Moreover, such a process is unique in distribution.  
\end{theo}

 
\section{The L\'{e}vy-It\^{o} Decomposition.}\label{sectionPLPLID}

\begin{assu}\label{generalAssumptionsLevyProcess}  Let $\ProbSpace$ be a complete probability space equipped with a filtration $\left\{ \mathcal{F}_{t} \right\}_{t \geq 0}$, that satisfies the usual conditions, i.e. it is right continuous and $\mathcal{F}_{0}$ contains all sets of $\Prob$-measure zero. 

We will consider a $\Phi'_{\beta}$-valued L\'{e}vy process $L=\left\{ L_{t} \right\}_{t\geq 0}$ and we assume that: \hfill
\begin{enumerate}
\item $L_{t}-L_{s}$ is independent of $\mathcal{F}_{s}$ for all $0\leq s <t$.
\item $\forall T > 0$, the family $\{ L_{t}: t \in [0,T] \}$ is equicontinuous as operators from $\Phi$ into $L^{0} \ProbSpace$.
\end{enumerate}
Then, it is a consequence of Theorem \ref{theoConditCadlagVersionLevyProc} that there  exists a countably Hilbertian topology $\vartheta_{L}$ on $\Phi$ such that $L$ is a $(\widetilde{\Phi_{\vartheta_{L}}})'_{\beta}$-valued c\`{a}dl\`{a}g process. 
We denote by $\Omega_{L} \subseteq \Omega$ a set with $\Prob (\Omega_{L})=1$ and such that for each $\omega \in \Omega_{L}$ the map $t \mapsto L_{t}(\omega)$ is c\`{a}dl\`{a}g in  $(\widetilde{\Phi_{\vartheta_{L}}})'_{\beta}$.
\end{assu}

It is important to remark that if $\Phi$ is a barrelled nuclear space, then Assumption \ref{generalAssumptionsLevyProcess}\emph{(2)} is satisfied by every $\Phi'_{\beta}$-valued L\'{e}vy process (see the proof of Corollary \ref{coroExistCadlagVersionLevyInBarrelledNuclearSpace}).

\subsection{Poisson Random Measures and Poisson Integrals.}\label{subsectionPRMPI}

In this section we study basic properties of the Poisson integrals defined by a stationary Poisson point process and its associated Poisson random measure on the dual of a nuclear space (see \cite{IkedaWatanabe}, Sections 1.8 and 1.9, for the basic definitions). For our proof of the L\'{e}vy-It\^{o} decomposition we will follow a program that can be thought as an infinite dimensional version of the arguments in  \cite{Bretagnolle:1973}, where the Poisson integrals will play a central role.  

Let $p=\{ p_{t} \}_{t \geq 0}$ be a $\{ \mathcal{F}_{t} \}$-adapted stationary Poisson point process on $( \Phi'_{\beta}, \mathcal{B}(\Phi'_{\beta}))$. Let $N$ be the Poisson random measure on $[0,\infty) \times \Phi'_{\beta}$ associated to $p$, i.e. 
\begin{equation} \label{poissonRandomMeasure}
N_{p}(t,A)(\omega) = \sum_{0 \leq s \leq t} \ind{A}{p_{s}(\omega)}, \quad \forall \omega \in \Omega, \, t \geq 0, \, A \in \mathcal{B}(\Phi'_{\beta}).
\end{equation}

As $p$ is stationary, there exists a Borel measure $\nu_{p}$ on $\Phi'_{\beta}$ such that 
\begin{equation} \label{relationNandNup}
\Exp ( N_{p}(t,A))= t \nu_{p}(A),  \quad \forall \, t \geq 0, \, A \in \mathcal{B}(\Phi'_{\beta}).
\end{equation}
We call $\nu_{p}$ the \emph{characteristic measure} of $p$.

Let $A \in \mathcal{B}(\Phi'_{\beta})$ with $\nu_{p}(A)< \infty$. For each $t \geq 0$ the \emph{Poisson integral with respect to $N_{p}$} is defined by 
\begin{equation} \label{poissonIntegralForIdenity}
J^{(p)}_{t}(A)(\omega) \defeq \int_{A} f N_{p}(t,df)(\omega)=  \sum_{0 \leq s \leq t} p_{s}(\omega) \ind{A}{p_{s}(\omega)}, \quad \forall \omega \in \Omega.    
\end{equation}

From now on we assume that $p=\{ p_{t} \}_{t \geq 0}$ is a regular process in $( \Phi'_{\beta}, \mathcal{B}(\Phi'_{\beta}))$. The following result presents the main properties of the Poisson integral process. 

\begin{prop}\label{propPropertiesPoissonIntegrals}
The process $J^{(p)}(A)=\{ J^{(p)}_{t}(A) \}_{t \geq 0}$  
is a $\{ \mathcal{F}_{t} \}$-adapted $\Phi'_{\beta}$-valued, regular, c\`{a}dl\`{a}g, L\'{e}vy process. For every $t \geq 0$ the distribution of $J^{(p)}_{t}(A)$ is given by
\begin{equation} \label{distributionPoissonIntegral}
\Prob \left( \omega : J^{(p)}_{t}(A)(\omega) \in \Gamma \right)= e^{-t\nu_{p}(A)} \sum_{k=0}^{\infty} \frac{t^{k}}{k!} (\restr{\nu_{p}}{A})^{\ast k}\left(\Gamma \right), \quad \forall \,  \Gamma \in \mathcal{B}(\Phi'_{\beta}). 
\end{equation} 
and its characteristic function is 
\begin{equation} \label{charactFunctionPoissonIntegral}
\Exp \left( \exp\left\{ i J^{(p)}_{t}(A)[\phi] \right\} \right) = \exp \left\{ t \int_{A} \left( e^{i f[\phi]}-1 \right) \nu_{p} (d f) \right\}, \quad \forall \, \phi \in \Phi.   
\end{equation}
Moreover if $\int_{A} \abs{f[\phi]} \nu_{p} (df) < \infty$ for each $\phi \in \Phi$, then
\begin{equation} \label{meanPoissonIntegral}
\Exp \left( J^{(p)}_{t}(A)[\phi] \right) = t \int_{A} f[\phi] \nu_{p} (df), \quad \forall \, \phi \in \Phi,
\end{equation}
Furthermore if $\int_{A} \abs{f[\phi]}^{2} \nu_{p}(df) < \infty$ for each $\phi \in \Phi$, then
\begin{equation} \label{variancePoissonIntegral}
\mbox{Var} \left( J^{(p)}_{t}(A)[\phi] \right) = t \int_{A} \abs{f[\phi]}^{2} \nu_{p} (df), \quad \forall \, \phi \in \Phi. 
\end{equation}
\end{prop}
\begin{prf}
The fact that $J^{(p)}(A)$ is a $\{ \mathcal{F}_{t} \}$-adapted c\`{a}dl\`{a}g regular process with independent and stationary increments is immediate from  \eqref{poissonIntegralForIdenity} and the corresponding properties of the processes $p$ and $\{ N_{p}(t,A) \}_{t \geq 0}$. It is clear from \eqref{poissonIntegralForIdenity} that $J^{(p)}_{0}(A)=0$ $\Prob$-a.e. 

The proofs of \eqref{distributionPoissonIntegral}, \eqref{charactFunctionPoissonIntegral}, \eqref{meanPoissonIntegral} and \eqref{variancePoissonIntegral} follow from  similar arguments to those used in the proofs of Theorems 2.3.7 and 2.3.9 in \cite{ApplebaumLPSC} where analogous results are proved for the case of Poisson integrals defined by the Poisson random measure of a $\R^{d}$-valued L\'{e}vy process. 

Finally let $G \in C_{b}(\Phi'_{\beta})$ and let $N>0$ such that $\sup_{f \in \Phi'} \abs{G(f)} \leq N$. Then, from \eqref{distributionPoissonIntegral} we have: 
\begin{align*}
\lim_{t \rightarrow 0+} \abs{ \int_{\Phi'} G(f) \mu_{J^{(p)}_{t}(A)}(df) - \int_{\Phi'} G(f) \delta_{0}(df)} 
& \leq \lim_{t \rightarrow 0+} \abs{ e^{-t \nu_{p}(A)}  \sum_{k=1}^{\infty} \frac{(t\nu_{p}(A)N)^{k}}{k!} } \\
& = \lim_{t \rightarrow 0+} e^{-t \nu_{p}(A)} \left( e^{-t \nu_{p}(A)N}- 1 \right)=0.   
\end{align*}
Then it follows that the map $t \mapsto \mu_{J^{(p)}_{t}(A)}$ is weakly continuous. Hence, $J^{(p)}(A)$ is a $\Phi'_{\beta}$-valued L\'{e}vy process.  
\end{prf}

Now if $\int_{A} \abs{f[\phi]} \nu_{p} (df) < \infty$ for each $\phi \in \Phi$, then for each $t \geq 0$ we define the \emph{compensated Poisson integral with respect to} $N_{p}$ by 
\begin{equation} \label{compensatedPoissonIntegralForIdenity}
\widetilde{J}^{(p)}_{t}(A)[\phi] \defeq \int_{A} f \widetilde{N}_{p}(t,df)[\phi] = \int_{A} f N_{p}(t,df)[\phi] - t \int_{A} f[\phi] \nu_{p}(df), \quad \forall  \, \phi \in \Phi.      
\end{equation}
The process  $\widetilde{J}^{(p)}(A)=\{ \widetilde{J}^{(p)}_{t}(A)  \}_{t \geq 0}$ is a $\Phi'_{\beta}$-valued, mean-zero, square integrable $\{ \mathcal{F}_{t} \}$-adapted regular c\`{a}dl\`{a}g L\'{e}vy process. In particular, for each $\phi \in \Phi$ the process $\widetilde{J}^{(p)}(A)[\phi]$ is a real-valued martingale. Moreover for each $t \geq 0$ it follows from \eqref{charactFunctionPoissonIntegral} and \eqref{variancePoissonIntegral} that    
\begin{equation} \label{charactFunctionCompensatedPoissonIntegral}
\Exp \left( \exp\left\{ i \widetilde{J}^{(p)}_{t}(A)[\phi] \right\} \right) = \exp \left\{ t \int_{A} \left( e^{i f[\phi]}-1-if[\phi] \right) \nu_{p} (d f) \right\}, \quad  \forall \, \phi \in \Phi. 
\end{equation}
Furthermore if $\int_{A} \abs{f[\phi]}^{2} \nu_{p}(df) < \infty$, for each $\phi \in \Phi$, then
\begin{equation} \label{secondMomentCompensatedPoissonIntegral}
\Exp \left( \abs{ \widetilde{J}^{(p)}_{t}(A)[\phi] }^{2} \right) = t \int_{A} \abs{f[\phi]}^{2} \nu_{p} (df),  \quad \forall \, \phi \in \Phi. 
\end{equation}

Other important properties of Poisson integrals are summarized in the following result. 

\begin{theo} \label{independencePoissonIntegrals} Let $A_{1}, A_{2} \in \mathcal{B}(\Phi'_{\beta})$ disjoint sets with $\nu_{p}(A_{1}), \nu_{p}(A_{2})< \infty$. Then the processes $J^{(p)}(A_{1})$ and $J^{(p)}(A_{2})$ are independent. If moreover $\int_{A_{i}} \abs{f[\phi]} \nu_{p} (df) < \infty$, for all $\phi \in \Phi$, $i=1,2$, then the processes $\widetilde{J}^{(p)}(A_{1})$ and $\widetilde{J}^{(p)}(A_{2})$ are independent. 
\end{theo}
\begin{prf}
Let $\phi_{1}, \dots, \phi_{n} \in \Phi$. Then it follows from \eqref{poissonIntegralForIdenity}  that the $\R^{n}$-valued stochastic processes $(J^{(p)}(A_{1})[\phi_{1}], \dots, J^{(p)}(A_{1})[\phi_{n}])$ and $(J^{(p)}(A_{2})[\phi_{1}], \dots, J^{(p)}(A_{2})[\phi_{n}])$ are compound Poisson processes whose jumps occurs at distinct times for each $\omega \in \Omega$ due to the fact that $A_{1}$ and $A_{2}$ are disjoint. Then the same arguments used in the proof of Theorem 2.4.6 of \cite{ApplebaumLPSC} p.116 show that the processes $(J^{(p)}(A_{1})[\phi_{1}], \dots, J^{(p)}(A_{1})[\phi_{n}])$ and $(J^{(p)}(A_{2})[\phi_{1}], \dots, J^{(p)}(A_{2})[\phi_{n}])$ are independent. Then, as the processes $J^{(p)}(A_{1})$ and $J^{(p)}(A_{2})$ are regular, it follows from Proposition \ref{independenceStochProcessDualSpace} that they are independent. 

Now if the integrability condition $\int_{A_{i}} \abs{f[\phi]} \nu_{p} (df) < \infty$, for all $\phi \in \Phi$, $i=1,2$, is satisfied, the independence of $\widetilde{J}^{(p)}(A_{1})$ and $\widetilde{J}^{(p)}(A_{2})$ follows immediately from the independence of $J^{(p)}(A_{1})$ and $J^{(p)}(A_{2})$. 
\end{prf}

\subsection{The Poisson random measure and Poisson integrals of  a L\'{e}vy process}\label{subSectionPMPILP}

For the L\'{e}vy process $L=\left\{ L_{t} \right\}_{t\geq 0}$, we define by $\Delta L_{t} \defeq L_{t}-L_{t-}$ the \emph{jump} of the process $L$ at the time $t\geq 0$. Note that from Assumption \ref{generalAssumptionsLevyProcess} we have that $\Delta L= \{ \Delta L_{t} \}_{t \geq 0}$ is a $\{ \mathcal{F}_{t} \}$-adapted $\Phi'_{\beta}$-valued regular stochastic process.   

We say that a set $A \in \mathcal{B}(\Phi'_{\beta} \setminus \{ 0\})$ is \emph{bounded below} if $0 \notin \overline{A}$, where $\overline{A}$ is the closure of $A$.  Then $A$ is bounded below if and only if $A$ is contained in the complement of a neighborhood of zero. We denote by $\mathcal{A}$ the collection of all the subsets of $\Phi'_{\beta} \setminus \{0\}$ that are bounded below. Clearly $\mathcal{A}$ is a ring. 

For $A \in \mathcal{B}( \Phi'_{\beta} \setminus \{ 0\})$ and $t \geq 0$ define
\begin{equation*} 
N(t,A)(\omega)=\# \left\{ 0 \leq s \leq t \st \Delta L_{s}(\omega) \in A \right\} = \sum_{0 \leq s \leq t} \ind{A}{\Delta L_{s}(\omega)}, \quad \mbox{if } \omega \in \Omega_{L}
\end{equation*}
and $N(t,A)(\omega)=0$ if $\omega \in \Omega_{L}^{c}$. 

We have from our Assumption \ref{generalAssumptionsLevyProcess}(2) that $L$ is a 
$(\widetilde{\Phi_{\vartheta_{L}}})'_{\beta}$-valued c\`{a}dl\`{a}g process. Then, the fact that $\widetilde{\Phi_{\vartheta_{L}}}$ is an ultrabornological space, Proposition 3.10 in \cite{FonsecaMora:2018} and the arguments in the proof of Proposition 3.3 in \cite{FonsecaMora:2018} (see Remark 3.9 in \cite{FonsecaMora:2018}), imply that for every $\omega \in \Omega_{L}$ and $t \geq 0$, there exists a continuous Hilbertian semi-norm $\varrho=\varrho(\omega,t)$ on $\Phi$ such that the map $s \mapsto L_{s}(\omega)$ is c\`{a}dl\`{a}g from $[0,t]$ into the Hilbert space $\Phi'_{\varrho}$. But as $\Phi'_{\varrho}$ is a complete separable metric space, the above implies that $\Delta L_{s}(\omega) \neq 0$ for a finite number of $s \in [0,t]$. Hence $A \mapsto N(t,A)(\omega)$ is a counting measure on $\left( \Phi'_{\beta} \setminus \{ 0 \}, \mathcal{B}(\Phi'_{\beta} \setminus \{ 0\}) \right)$. Then $\Delta L= \{ \Delta L_{t} \}_{t \geq 0}$ is a regular stationary Poisson point processes on $\left( \Phi'_{\beta} \setminus \{ 0 \}, \mathcal{B}(\Phi'_{\beta} \setminus \{ 0\}) \right)$ and $N=\{N(t,A): \, t \geq 0, A \in \mathcal{B}(\Phi'_{\beta} \setminus \{ 0\})\}$ is the \emph{Poisson random measure} associated to $\Delta L$ with respect to the ring $\mathcal{A}$. 
Let $\nu$ be the characteristic measure of $\Delta L$, i.e. the Borel measure on $\Phi'_{\beta}$ defined by $\nu( \{ 0 \})=0$ and that satisfies:
\begin{equation} \label{intensityMeasurePoissonRandomMeasure}
\Exp \left( N(t,\Gamma)\right) = t \nu(\Gamma), \quad \forall \, t \geq 0, \, \Gamma \in \mathcal{B}\left(\Phi'_{\beta} \setminus \{0\}\right).
\end{equation}  
Clearly $\nu(A)< \infty$ for every $A \in \mathcal{A}$. 

\begin{defi} \label{defiRegularMeasure}
Let $\mu$ be a Borel measure on $\Phi'_{\beta}$. We will say that $\mu$ is a $\theta$-\emph{regular} measure on $\Phi'_{\beta}$ if there exists a weaker countably Hilbertian topology $\theta$ on $\Phi$ such that $\mu$ is concentrated on $\Phi'_{\theta}$, i.e. $\mu(\Phi'_{\beta} \setminus \Phi'_{\theta})=0$. 
\end{defi}

\begin{lemm} \label{lemmRestrLevyMeasureIsRadon}
The measure $\nu$ is $\theta_{L}$-regular (where $\theta_{L}$ is as in Assumption \ref{generalAssumptionsLevyProcess}). Moreover for every $A \in \mathcal{B}(\Phi'_{\beta})$ such that $\nu(A)< \infty$ (in particular if $A \in \mathcal{A}$), $\restr{\nu}{A}$ is $\theta_{L}$-regular and $\restr{\nu}{A} \in \goth{M}_{R}^{b}(\Phi'_{\beta})$. 
\end{lemm}
\begin{prf}
First, note that from Assumption \ref{generalAssumptionsLevyProcess}(2) we have that $\Delta L_{t} \in \Phi'_{\theta_{L}}$ $\forall t \geq 0$ $\Prob$-a.e. and hence from \eqref{intensityMeasurePoissonRandomMeasure} we have that $\nu \left( \Phi'_{\beta} \setminus \Phi'_{\theta_{L}}\right)=0$ and hence $\nu$ is $\theta_{L}$-regular. 

Now, let $A \in \mathcal{B}(\Phi'_{\beta})$ such that $\nu(A)< \infty$. Because the measure $\nu$ is $\theta_{L}$-regular so is $\restr{\nu}{A}$. If we consider the canonical $\Phi'_{\beta}$-valued random variable $X_{\nu,A}$ whose probability distribution is $\displaystyle{\restr{\nu}{A}(\cdot) \, / \, \restr{\nu}{A}(\Phi'_{\theta_{L}})}$, we then have that $\Prob (X_{\nu,A} \in \Phi'_{\theta_{L}})=1$ and hence $X_{\nu,A}$ is a regular random variable. Therefore, Theorem \ref{theoCharacterizationRegularRV} shows that the probability distribution of $X_{\nu,A}$ is a Radon measure on $\Phi'_{\beta}$. Then, $\restr{\nu}{A} \in \goth{M}_{R}^{b}(\Phi'_{\beta})$. 
\end{prf}

For every $A \in \mathcal{B}(\Phi'_{\beta})$ such that $\nu(A)< \infty$, we will denote by $J(A)$ the Poisson integral with respect to $N$ and if $\int_{A} \abs{f[\phi]}^{2} \nu(df) < \infty$, for each $\phi \in \Phi$, we denote by $\widetilde{J}(A)$ the compensated Poisson integral with respect to $N$.  

\begin{theo} \label{independenceLevyAndItsPoissonIntegrals} 
Let $A \in \mathcal{B}(\Phi'_{\beta})$ with $\nu(A)< \infty$. Then $ L-J(A) = \left\{ L_{t} -  J_{t}(A) \right\}_{t \geq 0}$ is a $\Phi'_{\beta}$-valued L\'{e}vy  process. Moreover the processes $L- J(A)$ and $J(A)$ are independent. 
\end{theo}
\begin{prf}
First, similar arguments to those used in the proof of Theorem 2.4.8 of  \cite{ApplebaumLPSC} for the case of $\R^{n}$-valued L\'{e}vy processes show that $L-J(A)$ is a $\Phi'_{\beta}$-valued L\'{e}vy process. To prove the independence of $L- J(A)$ and $J(A)$, let $\phi_{1}, \dots, \phi_{n} \in \Phi$. As $\left((L-J(A))[\phi_{1}], \dots, (L-J(A))[\phi_{n}] \right)$ and $\left( J(A)[\phi_{1}], \dots, J(A)[\phi_{n}] \right)$ are $\R^{n}$-valued L\'{e}vy processes that have their jumps at distinct times for each $\omega \in \Omega$, the same arguments used in the proof of Lemma 7.9 and Theorem 7.12 of \cite{Medvegyev} p.468-71 show that the processes $\left((L-J(A))[\phi_{1}], \dots, (L-J(A))[\phi_{n}] \right)$ and $\left( J(A)[\phi_{1}], \dots, J(A)[\phi_{n}] \right)$ are independent. Then the independence of $L-J(A)$ and $J(A)$ follows from Proposition \ref{independenceStochProcessDualSpace} as both $L-J(A)$ and $J(A)$ are regular processes.  
\end{prf}

\subsection{L\'{e}vy Measures on the Dual of a Nuclear Space}\label{subSectionLM}

L\'{e}vy measures play an important role on the study of L\'{e}vy processes and infinitely divisible measures. In this section we introduce our definition of L\'{e}vy measure and present some of its basic properties.  

\begin{defi} \label{defiLevyMeasureDualNuclearSpace}
A Borel measure $\lambda$ on $\Phi'_{\beta}$ is a L\'{e}vy measure if 
\begin{enumerate}
\item $\lambda (\{ 0 \})=0$, 
\item for each neighborhood of zero $U \subseteq \Phi'_{\beta}$, $\restr{\lambda}{U^{c}}  \in  \goth{M}^{b}_{R}(\Phi'_{\beta})$, 
\item there exists a continuous Hilbertian semi-norm $\rho$ on $\Phi$ such that 
\begin{equation} \label{integrabilityPropertyLevyMeasure}
\int_{B_{\rho'}(1)} \rho'(f)^{2} \lambda (df) < \infty,  \quad \mbox{and} \quad  \restr{\lambda}{B_{\rho'}(1)^{c}} \in \goth{M}^{b}_{R}(\Phi'_{\beta}), 
\end{equation}
where we recall that $B_{\rho'}(1) \defeq  \{f \in \Phi': \rho'(f) \leq 1\} = B_{\rho}(1)^{0}$. 

\end{enumerate}
\end{defi}
Note that \eqref{integrabilityPropertyLevyMeasure} implies that 
\begin{equation} \label{equDefLevyMeasure}
\int_{\Phi'} (\rho'(f)^{2} \wedge 1) \lambda (df) < \infty,
\end{equation}
which resembles the property that characterizes L\'{e}vy measures on Hilbert spaces (see \cite{Parthasarathy}).

\begin{prop} \label{propLevyMeasuresAreRadon}
Every L\'{e}vy measure on $\Phi'_{\beta}$ is a $\sigma$-finite Radon measure. Moreover if $\Phi$ is a barrelled nuclear space, every L\'{e}vy measure on $\Phi'_{\beta}$ is $\theta$-regular.
\end{prop}
\begin{prf}
Let $\lambda$ be a L\'{e}vy measure on $\Phi'_{\beta}$ and let $\rho$ as in Definition \ref{defiLevyMeasureDualNuclearSpace}(3). From \eqref{equDefLevyMeasure} and standard arguments we have that  
$\lambda( B_{\rho'}(\epsilon)^{c})< \infty$ $\forall \, 0 <\epsilon \leq 1$. But the above together with $\lambda (\{ 0 \})=0$ imply that $\lambda$ is $\sigma$-finite. 

To prove that $\lambda$ is Radon, note that because $\restr{\lambda}{B_{\rho'}(1)^{c}} \in \goth{M}^{b}_{R}(\Phi'_{\beta})$, it is enough to show that $\restr{\lambda}{B_{\rho'}(1)} \in \goth{M}^{b}_{R}(\Phi'_{\beta})$. To show this, let $q: \Phi \rightarrow \R$ defined by 
$$q(\phi)^{2}=\int_{B_{\rho'}(1)} \abs{f[\phi]}^{2} \lambda(df), \quad \forall \, \phi \in \Phi.$$ 
It is clear that $q$ is a Hilbertian semi-norm on $\Phi$. Moreover because $q(\phi)^{2} \leq C \rho(\phi)^{2}$ for all $\phi \in \Phi$, where $C=\int_{B_{\rho'}(1)} \rho'(f)^{2} \lambda (df) < \infty$, then $q$ is continuous on $\Phi$. 

Now note that for every $\phi \in \Phi$ we have 
\begin{equation*}
1-\mbox{Re} \, \widehat{ \restr{\lambda}{B_{\rho'}(1)} }(\phi) 
= \int_{B_{\rho'}(1)} (1-\cos f[\phi]) \lambda (df) \leq \frac{1}{2} \int_{B_{\rho'}(1)} {f[\phi]}^{2} \lambda (df)= \frac{1}{2} q(\phi)^{2}.
\end{equation*}
Then it follows that $\widehat{ \restr{\lambda}{B_{\rho'}(1)} }$ is continuous on $\Phi$. Finally by Minlos' theorem (see \cite{DaleckyFomin}, Theorem III.1.3, p.88) this shows that $ \restr{\lambda}{B_{\rho'}(1)} $ is a Radon measure on $\Phi'_{\beta}$. Therefore $\lambda \in \goth{M}^{b}_{R}(\Phi'_{\beta})$. 

Finally if $\Phi$ is a barrelled nuclear space, the fact that $\lambda$ is $\theta$-regular is a consequence of Theorem \ref{theoCharacterizationRegularRV} and the fact that $\lambda$ is a Radon measure on $\Phi'_{\beta}$.   
\end{prf}

\subsection{The L\'{e}vy Measure of a L\'{e}vy process}\label{subSectionLMLP}

We proceed to show that the measure $\nu$ associated to the Poisson measure $N$ of the L\'{e}vy process $L$ is a L\'{e}vy measure on $\Phi'_{\beta}$.  We start by recalling the concept of Poisson measure that will be of great importance for our 	forthcoming arguments. 

Let $\mu \in \goth{M}_{R}^{b}(\Phi'_{\beta})$. The measure $e(\mu) \in \goth{M}_{R}^{1}(\Phi'_{\beta})$ defined by 
$$ e(\mu)(\Gamma)=e^{-\mu(\Phi'_{\beta})} \sum_{k=0}^{\infty} \frac{1}{k!} \mu^{\ast k}(\Gamma), \quad \forall \, \Gamma \in \mathcal{B}(\Phi'_{\beta}),$$ 
is called a \emph{Poisson measure}. We call $\mu$ the \emph{Poisson exponent} of $e(\mu)$. It is clear that $e(\mu)$ is infinitely divisible and that  
\begin{equation} \label{charactFunctPoissonMeas}
\widehat{e(\mu)}(\phi)=\exp \left[ -(\widehat{\mu}(0)-\widehat{\mu}(\phi))\right], \quad \forall \, \phi \in \Phi.
\end{equation}
Observe  that for $\mu \in \goth{M}_{R}^{b}(\Phi'_{\beta})$, $\abs{\widehat{e(\mu)}(\phi)}^{2}$ is the characteristic function of a measure belonging to $\goth{M}_{R}^{b}(\Phi'_{\beta})$. Indeed, it is the characteristic function of the measure $e(\mu+\overline{\mu})=e(\mu) \ast e(\overline{\mu})$, where $\overline{\mu} \in \goth{M}_{R}^{b}(\Phi'_{\beta})$ is defined by $\overline{\mu}(\Gamma)=\mu(-\Gamma)$ for all $\Gamma \in \mathcal{B}(\Phi'_{\beta})$. 

Now to show that $\nu$ is a L\'{e}vy measure we will need two preliminary results. The following is a mild generalization of a result due to Fernique for the characteristic function of infinitely divisible measures on the space of distributions $\mathscr{D}'$. Its proof easily extends to our case so we omit it and refer the reader to \cite{Fernique:1967}, Corollaire 2.  

\begin{lemm}\label{ferniqueLemma}
Let $\mu$ be an infinitely divisible measure on $\Phi'_{\beta}$. 
Then, for every continuous Hilbertian seminorm $p$ on $\Phi$ and every $\epsilon \in \, ]0,\frac{1}{4}]$ such that:
$$ \forall \, \phi \in \Phi, \quad p(\phi) \leq 1 \quad \Rightarrow \quad \abs{ 1- \widehat{\mu}(\phi)} < \epsilon, $$
we have that 
$$ \forall \, \phi \in \Phi, \, \forall \, n \in \N, \quad n \cdot (1- \emph{Re} \, \widehat{\mu}^{1/n}(\phi) ) \leq 8 \epsilon (1+p(\phi)^{2}).$$  
\end{lemm}

Another result that will be of great importance for our forthcoming arguments is the following version of Minlos' lemma due to Fernique. With some modifications, its proof can be carried out as the proof of Lemme 2 in \cite{Fernique:1967-2} for bounded measures on $\mathscr{D'}$.  

\begin{lemm}[Minlos' lemma]\label{minlosLemma}
Let $\mu \in \goth{M}^{b}_{R}(\Phi'_{\beta})$. Suppose that there exists $\epsilon >0$ and a continuous Hilbertian seminorm $p$ on $\Phi$ such that 
$$ 1 - \emph{Re} \, \widehat{\mu}(\phi) 	\leq \epsilon (1+p(\phi)^{2}), \quad \forall \, \phi \in \Phi.$$
If $q$ is any continuous Hilbertian seminorm on $\Phi$, satisfying $p \leq q$ and such that $i_{p,q}$ is Hilbert-Schmidt, then we have that 
$$ \int_{\Phi'} (q'(f)^{2} \wedge 1) \mu (df) \leq \epsilon \left( 1+ \norm{i_{p,q}}^{2}_{\mathcal{L}_{2}(\Phi_{q},\Phi_{p})} \right) < \infty. $$
\end{lemm}

We are ready for the main result of this section:

\begin{theo} \label{associatedMeasureIsLevyMeasure}
The measure $\nu$ of the $\Phi'_{\beta}$-valued L\'{e}vy process $L$ is a L\'{e}vy measure on $\Phi'_{\beta}$.  
\end{theo}
\begin{prf}
By definition $\nu (\{0 \})=0$. Now, because for every neighborhood of zero $U \subseteq \Phi'_{\beta}$, we have that $U^{c} \in \mathcal{A}$, then $\restr{\nu}{U^{c}}  \in  \goth{M}_{R}^{b}(\Phi'_{\beta})$ (Lemma \ref{lemmRestrLevyMeasureIsRadon}). Therefore, it only remains to show that there exits a continuous Hilbertian semi-norm $\rho$ on $\Phi$ such that $\nu$ satisfies \eqref{equDefLevyMeasure} with $\lambda$ replaced by $\nu$. This is because \eqref{equDefLevyMeasure} implies that $\nu(B_{\rho'}(1)^{c})< \infty$ and hence from Lemma \ref{lemmRestrLevyMeasureIsRadon} we obtain that $\restr{\nu}{B_{\rho'}(1)^{c}} \in \goth{M}^{b}_{R}(\Phi'_{\beta})$.  
For our proof, we will benefit from some arguments of the proof of Lemma 2.1 in \cite{Dettweiler:1976}.  

Let $\goth{B}$ be a local base of closed neighborhoods of zero for $\Phi'_{\beta}$ and let $\mathcal{A}_{\goth{B}}=\{ V^{c}: V \in \goth{B}\}$. Because $\Phi'_{\beta}$ is Hausdorff, it follows that $\Phi'_{\beta} \setminus \{0\}= \bigcup_{A \in \mathcal{A}_{\goth{B}}} A$.  

For each $A \in \mathcal{A}_{\goth{B}}$, let $\nu_{A} \defeq \restr{\nu}{A}$. As each $A \in \mathcal{A}_{\goth{B}}$ satisfies $A \in \mathcal{A}$, we have that $\nu_{A} \in \goth{M}_{R}^{b}(\Phi'_{\beta})$ for all $A \in \mathcal{A}_{\goth{B}}$ (Lemma \ref{lemmRestrLevyMeasureIsRadon}). Now consider on $\mathcal{A}_{\goth{B}}$ the order relationship given by the inclusion of sets. Then $\{ \nu_{A} \}_{A \in \mathcal{A}_{\goth{B}}}$ is an increasing net (setwise) in $\goth{M}_{R}^{b}(\Phi'_{\beta})$. Moreover because $\mathcal{A}_{\goth{B}}$ is an increasing net of open subsets that satisfies $\Phi'_{\beta} \setminus \{ 0 \} = \bigcup_{A \in \mathcal{A}_{\goth{B}}} A$, and $\nu$ can be reduced to be a Borel measure on the (separable and metrizable) subspace $\Phi'_{\theta_{L}}$ of $\Phi'_{\beta}$ (this follows from Assumption \ref{generalAssumptionsLevyProcess}(2) and \eqref{intensityMeasurePoissonRandomMeasure}), it follows that $\nu = \sup_{A \in \mathcal{A}_{\goth{B}}} \nu_{A}$ (setwise) (see \cite{BogachevMT}, Propositions 7.2.2 and 7.2.5).  

On the other hand,  note that from Theorem \ref{independenceLevyAndItsPoissonIntegrals}, for each $A \in \mathcal{A}_{\goth{B}}$, the processes $L-J(A)$ and $J(A)$ are independent. Therefore we have 
\begin{equation} \label{decompChartFuncLtForPoissonInteg}
\widehat{\mu}_{L_{t}}(\phi)=\widehat{\mu}_{L_{t}-J_{t}(A)}(\phi) \cdot \widehat{ \mu}_{J_{t}(A)}(\phi), \quad \forall \, A \in \mathcal{A}_{\goth{B}}, \, t \geq 0, \phi \in \Phi. 
\end{equation}
Now for fixed  $A \in \mathcal{A}_{\goth{B}}$, $t \geq 0$, $\phi \in \Phi$, because $\abs{\widehat{\mu}_{L_{t}-J_{t}(A)}(\phi)} \leq 1$ it follows from \eqref{decompChartFuncLtForPoissonInteg} that 
$\abs{\widehat{\mu}_{L_{t}}(\phi)}^{2} \leq \abs{\widehat{ \mu}_{J_{t}(A)}(\phi)}^{2}  \leq 1$. Therefore we have that 
\begin{equation} \label{inequaForEquiContPoissInteg1}
1-\abs{\widehat{ \mu}_{J_{t}(A)}(\phi)}^{2} \leq 1- \abs{\widehat{\mu}_{L_{t}}(\phi)}^{2}, \quad \forall \, A \in \mathcal{A}_{\goth{B}}, \, t \geq 0, \phi \in \Phi. 
\end{equation}
On the other hand, note that if we take $t =1$ in \eqref{distributionPoissonIntegral} then we have $\mu_{J_{1}(A)}= e(\nu_{A})$, for all $A \in \mathcal{A}$. Therefore
it follows from \eqref{inequaForEquiContPoissInteg1} that 
\begin{equation} \label{inequaForEquiContPoissInteg2}
1-\abs{\widehat{ e(\nu_{A})}(\phi)}^{2} \leq 1- \abs{\widehat{\mu}_{L_{1}}(\phi)}^{2}, \quad \forall \, A \in \mathcal{A}_{\goth{B}}, \, \phi \in \Phi. 
\end{equation}
Now because $L_{1}$ is a regular random variable, it follows from Theorem \ref{theoCharacterizationRegularRV} that the map $\phi \mapsto L_{t}[\phi]$ from $\Phi$ into $L^{0}\ProbSpace$ is continuous. But this in turn implies that $\widehat{\mu_{L_{1}}}$ and hence $\abs{\widehat{\mu_{L_{1}}}}^{2}$ is continuous at zero. Therefore there exists a continuous Hilbertian semi-norm $p$ on $\Phi$ such that 
\begin{equation} \label{contCharFuncL1}
\forall \, \phi \in \Phi, \quad p(\phi) \leq 1 \quad \Rightarrow \quad 1-\abs{\widehat{\mu}_{L_{1}}(\phi)}^{2} < \frac{1}{4}, \quad \forall \, \phi \in \Phi.
\end{equation}
Hence it follows from \eqref{inequaForEquiContPoissInteg2} and \eqref{contCharFuncL1} that 
\begin{equation} \label{inequaForEquiContPoissInteg3}
\forall \, A \in \mathcal{A}_{\goth{B}}, \,  \phi \in \Phi, \quad p(\phi) \leq 1 \quad \Rightarrow \quad 1-\abs{\widehat{ e(\nu_{A})}(\phi)}^{2} < \frac{1}{4}. 
\end{equation}
Let $\phi \in \Phi$. For every $A \in \mathcal{A}_{\goth{B}}$ and every $n \in \N$, 
from \eqref{charactFunctPoissonMeas} for the measure $\nu_{A}$,  we have  
$$ - \log \abs{\widehat{ e(\nu_{A})}(\phi)}^{2/n} = \frac{2}{n} \int_{\Phi'} (1- \cos f[x]) \nu_{A}(df) \leq \frac{4}{n} \nu_{A}(\Phi'_{\beta}) < \infty. $$
So for fixed $A \in \mathcal{A}_{\goth{B}}$, by choosing $n \in \N$ sufficiently large such that $ \nu_{A}(\Phi'_{\beta}) \leq \frac{n}{4}$, and by using the elementary inequality $\frac{t}{4} \leq 1-e^{-t}$ that is valid for $t \in [0,1]$, by taking $t= - \log \abs{\widehat{ e(\nu_{A})}(\phi)}^{2/n}$ we obtain that 
\begin{align} 
1- \mbox{Re} \, \widehat{\nu_{A}}(\phi) \label{inequaForEquiContPoissInteg4}
& =  \int_{\Phi'} (1- \cos f[x]) \nu_{A}(df)  \\
& =  - \frac{n}{2} \log \abs{\widehat{ e(\nu_{A})}(\phi)}^{2/n} 
\leq 2 n \cdot  \left( 1- \abs{\widehat{ e(\nu_{A})}(\phi)}^{2/n} \right). \nonumber 
\end{align} 
On the other hand, from \eqref{inequaForEquiContPoissInteg3}
and Lemma \ref{ferniqueLemma} (with $\epsilon =\frac{1}{4}$) we have that 
\begin{equation} \label{inequaForEquiContPoissInteg5}
 n \cdot \left( 1- \abs{\widehat{ e(\nu_{A})}(\phi)}^{2/n} \right) \leq 2  (1+p(\phi)^{2}).
\end{equation}
Then \eqref{inequaForEquiContPoissInteg4} and \eqref{inequaForEquiContPoissInteg5} show that 
\begin{equation}
1- \mbox{Re} \, \widehat{\nu_{A}}(\phi) < 4 (1+p(\phi)^{2}), \quad \forall \, A \in \mathcal{A}_{\goth{B}}, \, \phi \in \Phi. 
\end{equation}
But from Lemma \ref{minlosLemma}, if $\rho$ is any continuous Hilbertian seminorm on $\Phi$, satisfying $p \leq \rho$ and such that $i_{p,\rho}$ is Hilbert-Schmidt, then the above arguments imply that  
$$ \int_{\Phi'} (\rho'(f)^{2} \wedge 1) \nu (df) = \sup_{A \in \mathcal{A}_{\goth{B}}} \int_{\Phi'} (\rho'(f)^{2} \wedge 1) \nu_{A} (df) \leq 4 \left( 1+ \norm{i_{p,\rho}}^{2}_{\mathcal{L}_{2}(\Phi_{\rho},\Phi_{p})} \right) < \infty. $$
Hence $\nu$ is a L\'{e}vy measure. 
\end{prf}

\begin{defi}
From now on, the measure $\nu$ of the L\'{e}vy process $L$ will be called the \emph{L\'{e}vy measure} of $L$. 
\end{defi}

\subsection{The L\'{e}vy-It\^{o} Decomposition.}\label{subsectionLID}

Our main objective for this section is to prove Theorem \ref{levyItoDecompositionTheorem}, which is the L\'{e}vy-It\^{o}  decomposition. We will need the following properties of the space of martingales taking values in the Hilbert space $\Phi'_{q}$. 

For a continuous Hilbertian semi-norm $q$ on $\Phi$ we denote by $\mathcal{M}^{2}(\Phi'_{q})$ and $\mathcal{M}^{2}_{T}(\Phi'_{q})$ the linear spaces of (equivalent clases of) $\Phi'_{q}$-valued mean-zero, square integrable, c\`{a}dl\`{a}g, $\{ \mathcal{F}_{t} \}$-adapted martingales defined respectively on $[0,\infty)$ and on $[0,T]$ (with $T>0$). 

The space $\mathcal{M}^{2}_{T}(\Phi'_{q})$, is a Banach space equipped with the norm $\norm{ \cdot }_{\mathcal{M}^{2}_{T}(\Phi'_{q})}$ defined by 
$$ \norm{ M }_{\mathcal{M}^{2}_{T}(\Phi'_{q})}= \left( \Exp \sup_{t \in [0,T]} q'(M_{t})^{2} \right)^{1/2}, \quad \forall \, M \in \mathcal{M}^{2}_{T}(\Phi'_{q}). $$ 

For every $T>0$, there exists a canonical inclusion $j_{T}$ of the space $\mathcal{M}^{2}(\Phi'_{q})$ into the space $\mathcal{M}^{2}_{T}(\Phi'_{q})$. Therefore we can equip $\mathcal{M}^{2}(\Phi'_{q})$ with the projective limit topology determined by the projective system $\{ (\mathcal{M}^{2}_{K}(\Phi'_{q}), j_{K}): K \in \N\}$. Then equipped with this topology, $\mathcal{M}^{2}(\Phi'_{q})$ is a Fr\'{e}chet space and a family of semi-norms generating its topology is $\{ \norm{ j_{K} (\cdot)}_{\mathcal{M}^{2}_{K}(\Phi'_{q})} : K \in \N\}$. In particular, convergence in $\mathcal{M}^{2}(\Phi'_{q})$ is then equivalent to convergence in the space $L^{2} \left( \Omega, \mathcal{F}, \Prob; \Phi'_{q} \right)$ uniformly on compact intervals of $[0,\infty)$.

Now we start with our preparations for the proof of Theorem \ref{levyItoDecompositionTheorem}. Let $\nu$ be the L\'{e}vy measure of $L$. According to Definition \ref{defiLevyMeasureDualNuclearSpace} and Theorem \ref{associatedMeasureIsLevyMeasure}, there exists a continuous Hilbertian semi-norm $\rho$ on $\Phi$ such that 
\begin{equation} \label{integrabilityLevyMeasureLevyProcess}
\int_{B_{\rho'}(1)} \rho'(f)^{2} \nu (df) < \infty,  \quad \mbox{and} \quad  \restr{\nu}{B_{\rho'}(1)^{c}} \in  \goth{M}_{R}^{b}(\Phi'_{\beta}), 
\end{equation}
where $B_{\rho'}(1) \defeq B_{\rho}(1)^{0}= \{f \in \Phi'_{\beta}: \rho'(f) \leq 1\}$. As $B_{\rho}(1)$ is a convex, balanced, neighborhood of zero, then its polar $B_{\rho'}(1)$ is a bounded, closed, convex, balanced subset of $\Phi'_{\beta}$.

\begin{theo} \label{existenceMartingaleLevyItoDecomp}
There exists a $\Phi'_{\beta}$-valued mean-zero, square integrable, c\`{a}dl\`{a}g L\'{e}vy process $M=\{ M_{t} \}_{t \geq 0}$ such that for all $t\geq 0$, it has characteristic function given by
\begin{equation} \label{charactFunctionMartingaleLevyItoDecomp}
 \Exp \left( e^{ i M_{t}[\phi] } \right) 
 = \exp \left\{ t \int_{B_{\rho'}(1)} \left( e^{i f[\phi]}-1-if[\phi] \right) \nu (d f) \right\}, \quad \forall \, \phi \in \Phi, 
\end{equation}
and second moments given by
\begin{equation} \label{secondMomentMartingaleLID}
\Exp \left( \abs{M_{t}[\phi]}^{2}\right) = t \int_{B_{\rho'}(1)} \abs{f[\phi]}^{2} \nu (df), \quad \forall \, \phi \in \Phi. 
\end{equation}
Moreover there exists a continuous Hilbertian semi-norm $q$ on $\Phi$, $\rho \leq q$, such that $i_{\rho,q}$ is Hilbert-Schmidt and for which $M$ is a $\Phi'_{q}$-valued mean-zero, square integrable, c\`{a}dl\`{a}g L\'{e}vy process with second moments given by
\begin{equation} \label{secondMomentHilbertSpaceMartingaleLID}
\Exp \left( q'(M_{t})^{2}\right) = t \int_{B_{\rho'}(1)} q'(f)^{2} \nu (df), \quad \forall \, t \geq 0. 
\end{equation}
\end{theo}
\begin{prf}
Let $\goth{B}$ be a local base of closed neighborhoods of zero for $\Phi'_{\beta}$. Let $\mathcal{A}_{\rho'}$ denotes the collection of all sets of the form $V \cap B_{\rho'}(1)$, where $V^{c} \in \goth{B}$. It is clear that  $\mathcal{A}_{\rho'} \subseteq \mathcal{A}$ (see Section \ref{subSectionPMPILP}). Moreover as $\Phi'_{\beta} \setminus \{0\}= \bigcup_{V \in \goth{B}} V^{c}$ (this follows because $\Phi'_{\beta}$ is Hausdorff) then we have $B_{\rho'}(1) \setminus \{0\}= \bigcup_{A \in \mathcal{A}_{\rho'}} A$. 

Fix an arbitrary $A \in \mathcal{A}_{\rho'}$. It follows from \eqref{integrabilityLevyMeasureLevyProcess} that 
\begin{equation} \label{momentsPoissonIntegALevyItoBoundedSeminormRho}
\int_{A} \abs{f[\phi]}^{2} \nu (df) \leq \rho(\phi)^{2} \int_{A} \rho'(f)^{2} \nu (df) \leq \rho(\phi)^{2} \int_{B_{p'}(1)} \rho'(f)^{2} \nu (df) < \infty, \quad \forall \, \phi \in \Phi.  
\end{equation}
Therefore the compensated Poisson integral $\widetilde{J}(A)$ is a $\Phi'_{\beta}$-valued mean-zero, square integrable, c\`{a}dl\`{a}g, regular L\'{e}vy process with characteristic function given by \eqref{charactFunctionCompensatedPoissonIntegral} and second moments given by \eqref{secondMomentCompensatedPoissonIntegral} (with $\widetilde{J}^{(p)}(A)$ replaced by $\widetilde{J}(A)$ and $\nu_{p}$ by $\nu$). Moreover, for each $\phi \in \Phi$ the process $\widetilde{J}(A)[\phi]$ is a real-valued $\{ \mathcal{F}_{t} \}$-adapted martingale. From Doob's inequality, \eqref{secondMomentCompensatedPoissonIntegral} and \eqref{momentsPoissonIntegALevyItoBoundedSeminormRho}, for every $T>0$ we have
$$ \Exp \left( \sup_{t \in [0,T] } \abs{\widetilde{J}_{t}(A)[\phi]}^{2}\right) \leq 4 \Exp \left( \abs{\widetilde{J}_{T}(A)[\phi]}^{2}\right) \leq C(T) \rho(\phi)^{2}, \quad \forall \, \phi \in \Phi, $$
where $C(T)=4T \int_{B_{\rho'}(1)} \rho'(f)^{2} \nu(df) < \infty$. Then from Theorem \ref{theoCondLevyCadlagVersHilbSpace}, there exists a continuous Hilbertian semi-norm $q$ on $\Phi$, $\rho \leq q$, such that $i_{\rho,q}$ is Hilbert-Schmidt and for which $\widetilde{J}(A)$ possesses a version that is a c\`{a}dl\`{a}g, mean-zero, square integrable, L\'{e}vy process in $\Phi'_{q}$. We denote this version again by $\widetilde{J}(A)$. Let $\{ \phi_{j}^{q} \}_{j \in \N} \subseteq \Phi$ be a complete orthonormal system in $\Phi_{q}$. Then from Fubini's theorem, Parseval's identity and \eqref{secondMomentCompensatedPoissonIntegral}, for every $t \geq 0$ we have 
\begin{equation} \label{secondMomentPoissonIntegAHilbertSpace}
\Exp \left( q'(\widetilde{J}_{t}(A))^{2}\right) = \sum_{j =1}^{\infty} \Exp \left( \abs{\widetilde{J}_{t}(A)[\phi_{j}^{q}]}^{2}\right) = t \sum_{j =1}^{\infty} \int_{A} \abs{f[\phi_{j}^{q}]}^{2} \nu (df) = t \int_{A} q'(f)^{2} \nu (df). 
\end{equation}

Now consider on $\mathcal{A}_{\rho'}$ the order induced by the inclusion of sets. Our next objective is to show that for every $T>0$ the net $\{ \{ \widetilde{J}_{t}(A) \}_{t \in [0,T]}: A \in \mathcal{A}_{\rho'} \}$ converges in the space 
$\mathcal{M}^{2}_{T}(\Phi'_{q})$. To do this, we will show that for a fixed $T>0$, $\{ \{ \widetilde{J}_{t}(A) \}_{t \in [0,T]}: A \in \mathcal{A}_{\rho'} \}$ is a Cauchy net in $\mathcal{M}^{2}_{T}(\Phi'_{q})$, then convergence follows by completeness of this space. 

Fix an arbitrary $T>0$. First observe that if $A_{1}, A_{2} \in \mathcal{A}_{\rho'}$, $A_{1} \subseteq A_{2}$, then from Doob's inequality, the definition of compensated Poisson integral and \eqref{secondMomentPoissonIntegAHilbertSpace} we have
\begin{equation} \label{cauchyCondPoissonIntegSetsALevyIto} 
\Exp \left( \sup_{t \in [0,T]} q'(\widetilde{J}_{t}(A_{1})-\widetilde{J}_{t}(A_{2}))^{2} \right) 
\leq  4 \Exp  \left( q'(\widetilde{J}_{T}(A_{2} \setminus A_{1}))^{2} \right) 
=  4 T \int_{A_{2} \setminus A_{1}} q'(f)^{2} \nu (df). 
\end{equation}
Therefore if we can show that 
\begin{equation} \label{limitIntegOnSetsAForCauchy}
\lim_{A \in \mathcal{A}_{\rho'}} \int_{A} q'(f)^{2} \nu (df) = \int_{B_{\rho'}(1)}  q'(f)^{2} \nu (df) < \infty,
\end{equation}
then \eqref{cauchyCondPoissonIntegSetsALevyIto} and \eqref{limitIntegOnSetsAForCauchy} would show that $\{ \widetilde{J}^{A} \}_{A \in \mathcal{A}_{\rho'}}$ is a Cauchy net on $\mathcal{M}^{2}_{T}(\Phi'_{q})$. 

To prove \eqref{limitIntegOnSetsAForCauchy}, note that as $\nu$ is a Borel measure on $B_{\rho'}(1)$, and $B_{\rho'}(1)$ is a Suslin set (it is the image under the continuous map $i'_{\rho}$ of the unit ball of the separable Hilbert space $\Phi'_{\rho}$), then $\nu$ is a Radon measure on $B_{\rho'}(1)$ (\cite{BogachevMT}, Vol II, Theorem 7.4.3, p.85). Moreover as $B_{\rho'}(1) \setminus \{ 0 \} = \bigcup_{A \in \mathcal{A}_{\rho'}} A$ and because $\nu$ is a Radon probability measure on $B_{\rho'}(1)$ such that $\nu(\{0\})=0$, we have that $\nu (B_{\rho'}(1)) = \lim_{A \in \mathcal{A}_{\rho'}} \nu(A)$ (see \cite{BogachevMT}, Vol. II, Propositions 7.2.2 and 7.2.5, p.74-5). Therefore from all the above we have 
\begin{eqnarray*}
\lim_{A \in \mathcal{A}_{\rho'}} \abs{ \int_{B_{\rho'}(1)}  q'(f)^{2} \nu (df) - \int_{A} q'(f)^{2} \nu (df) } 
& \leq &   \lim_{A \in \mathcal{A}_{\rho'}} \int_{B_{\rho'}(1) \setminus A} q'(f)^{2} \nu (df) \\
& \leq & \sup_{f \in B_{\rho'}(1)} q'(f)^{2} \lim_{A \in \mathcal{A}_{\rho'}} \mu ( B_{\rho'}(1) \setminus A)=0,  
\end{eqnarray*}
and hence \eqref{limitIntegOnSetsAForCauchy} is valid. 

Thus $\{ \{ \widetilde{J}_{t}(A) \}_{t \in [0,T]}: A \in \mathcal{A}_{\rho'} \}$ is a Cauchy net on $\mathcal{M}^{2}_{T}(\Phi'_{q})$ for every $T>0$. This in turn implies that $\{ \widetilde{J}^{A}: A \in \mathcal{A}_{\rho'} \}$ converges in $\mathcal{M}^{2}(\Phi'_{q})$. 
Therefore there exists some $M=\{ M_{t} \}_{t \geq 0}$ that is a $\Phi'_{q}$-valued mean-zero, square integrable, \cadlag martingale such that the net $\{ \widetilde{J}(A): A \in \mathcal{A}_{\rho'} \}$ converges to $M$ in $L^{2} \left( \Omega, \mathcal{F}, \Prob; \Phi'_{q} \right)$  uniformly on compact intervals of $[0,\infty)$. This uniform convergence, \eqref{secondMomentPoissonIntegAHilbertSpace} and \eqref{limitIntegOnSetsAForCauchy} imply that $M$ satisfies \eqref{secondMomentHilbertSpaceMartingaleLID}. Moreover viewing $M$ as a $\Phi'_{\beta}$-valued processes it is also a $\Phi'_{\beta}$-valued, mean-zero, square integrable, \cadlag martingale. 

To prove \eqref{charactFunctionMartingaleLevyItoDecomp} and \eqref{secondMomentMartingaleLID}, let $\phi \in \Phi$ arbitrary but fixed. From a basic estimate of the complex exponential function (proved in e.g. \cite{Sato}, Lemma 8.6, p.40) we have
$$ \abs{ e^{i f[\phi]}-1-if[\phi] } \leq \frac{\abs{f[\phi]}^{2}}{2}  \leq \frac{\rho(\phi)^{2} \rho'(f)^{2}}{2} \leq \frac{\rho(\phi)^{2}}{2}  < \infty, \quad \forall \, f \in B_{\rho'}(1).$$
Therefore the functions $f \mapsto (e^{i f[\phi]}-1-if[\phi])$ and $f \mapsto \abs{f[\phi]}^{2}$ are bounded on $B_{\rho'}(1)$. Then, using similar arguments to those used to prove \eqref{limitIntegOnSetsAForCauchy} we can show that 
\begin{equation} \label{limitWeakSecondMomentMartingLID}
\lim_{A \in \mathcal{A}_{\rho'}} \int_{A} \abs{f[\phi]}^{2} \nu (df) = \int_{B_{\rho'}(1)} \abs{f[\phi]}^{2} \nu (df),
\end{equation}
and 
\begin{equation} \label{limitCharactFuncMartingLID}
\lim_{A \in \mathcal{A}_{\rho'}} \int_{A} (e^{i f[\phi]}-1-if[\phi]) \nu (df) = \int_{B_{\rho'}(1)} (e^{i f[\phi]}-1-if[\phi]) \nu (df).
\end{equation}
On the other hand, for any $A \in \mathcal{A}_{\rho'}$ and $T>0$, we have that  
\begin{equation} \label{uniformConvergL2PoissonIntegDualSpace} 
\Exp \left( \sup_{t \in [0,T]} \abs{M_{t}[\phi]-\widetilde{J}_{t}(A)[\phi]}^{2} \right) 
\leq  q(\phi)^{2} \, \Exp \left( \sup_{t \in [0,T]} q'( M_{t}-\widetilde{J}_{t}(A))^{2} \right).
\end{equation}
Therefore the fact that $\{ \widetilde{J}(A): A \in \mathcal{A}_{\rho'} \}$ converges to $M$ in $\mathcal{M}^{2}(\Phi'_{q})$ and \eqref{uniformConvergL2PoissonIntegDualSpace}, imply that $\{ \widetilde{J}(A)[\phi]: A \in \mathcal{A}_{\rho'} \}$ converges to $M[\phi]$ in $L^{2} \left( \Omega, \mathcal{F}, \Prob\right)$  uniformly on compact intervals of $[0,\infty)$. This convergence together with \eqref{secondMomentCompensatedPoissonIntegral} and \eqref{limitWeakSecondMomentMartingLID} imply \eqref{secondMomentMartingaleLID}. 

Furthermore as for each $t \geq 0$, $\left\{ \widetilde{J}_{t}(A)[\phi]: A \in \mathcal{A}_{\rho'} \right\}$ converges to $M_{t}[\phi]$ in $L^{2} \left( \Omega, \mathcal{F}, \Prob\right)$, then the net of characteristics functions $\{ \Exp \left( \exp \left( i \widetilde{J}_{t}(A)[\phi] \right) \right): A \in \mathcal{A}_{\rho'} \}$ converges to the characteristic function $\Exp \left( \exp \left(i M_{t}[\phi]\right) \right)$ of $M_{t}$. Then, \eqref{charactFunctionCompensatedPoissonIntegral} and \eqref{limitCharactFuncMartingLID} implies \eqref{charactFunctionMartingaleLevyItoDecomp}.     

Finally as $\mathcal{M}^{2}(\Phi'_{q})$ is metrizable, we can choose a subsequence $\{ \widetilde{J}^{A_{n}}: n \in \N \}$ that converges to $M$ in $\mathcal{M}^{2}(\Phi'_{q})$. Then, $\{ \widetilde{J}^{A_{n}}: n \in \N \}$ converges to $M$ in $L^{2} \left( \Omega, \mathcal{F}, \Prob; \Phi'_{q} \right)$  uniformly on compact intervals of $[0,\infty)$ and because each $\widetilde{J}^{A_{n}}$ is a $\Phi'_{q}$-valued L\'{e}vy process, this implies that $M$ is also a $\Phi'_{q}$-valued L\'{e}vy process. This last fact implies that $M$ is also a $\Phi'_{\beta}$-valued L\'{e}vy process.      
\end{prf}

\begin{nota}
We denote by $\left\{ \int_{B_{\rho'}(1)} f \widetilde{N}(t,df): t \geq 0 \right\}$ the process $M=\{ M_{t} \}_{t \geq 0}$ defined in Theorem \ref{existenceMartingaleLevyItoDecomp}. 
\end{nota}

The next result follows from Proposition \ref{propPropertiesPoissonIntegrals} and because $\restr{\nu}{B_{\rho'}(1)^{c}} \in  \goth{M}_{R}^{b}(\Phi'_{\beta})$. 

\begin{prop} \label{propExistenceLargeJumpPartLID}
The $\Phi'_{\beta}$-valued process $\left\{ \int_{B_{\rho'}(1)^{c}} f N(t,df): t \geq 0  \right\}$ defined by
\begin{equation} \label{definitionCompoundPoissonLID}
 \int_{B_{\rho'}(1)^{c}} f N(t,df) (\omega)[\phi]= \sum_{0 \leq s \leq t} \Delta L_{s}(\omega) [\phi] \ind{B_{\rho'}(1)^{c} }{\Delta L_{s}(\omega)}, \quad \forall \omega \in \Omega, \, \phi \in \Phi, \, t \geq 0.   
\end{equation}
is a $\{ \mathcal{F}_{t} \}$-adapted $\Phi'_{\beta}$-valued, regular, c\`{a}dl\`{a}g L\'{e}vy process. 
Moreover $\forall \phi \in \Phi$, $t \geq 0$
\begin{equation} \label{charactFunctionCompoundPoissonLID}
\Exp \left( \exp\left\{ i \int_{B_{\rho'}(1)^{c}} f N(t,df)[\phi] \right\} \right) = \exp \left\{ t \int_{B_{\rho'}(1)^{c}} \left( e^{i f[\phi]}-1 \right) \nu (d f) \right\}.
\end{equation}
\end{prop}

Now define the process $Y=\{ Y_{t} \}_{t\geq 0}$ by
\begin{equation} \label{auxiliarProcess1LevyItoDecomp}
Y_{t}=L_{t}-\int_{B_{\rho'}(1)^{c}} f N(t,df), \quad \forall \, t \geq 0. 
\end{equation}

From Theorem \ref{independenceLevyAndItsPoissonIntegrals} and Proposition \ref{propExistenceLargeJumpPartLID} it follows that $Y$ is  a $\{ \mathcal{F}_{t} \}$-adapted $\Phi'_{\beta}$-valued regular c\`{a}dl\`{a}g L\'{e}vy process independent of $\left\{ \int_{B_{\rho'}(1)^{c}} f N(t,df): t \geq 0 \right\}$.  Moreover from the definition of the Poisson integral \eqref{definitionCompoundPoissonLID}, for any $0 \leq s<t$, 
$$Y_{t}-Y_{s} = L_{t}-L_{s}- \sum_{s < u \leq t} \Delta L_{u} \ind{B_{\rho'}(1)^{c}}{\Delta L_{u}}.$$ 	
Therefore $ \sup_{t \geq 0} \rho'(\Delta Y_{t}(\omega)) \leq 1$ for each $\omega \in \Omega$. This in particular implies that for each $ \phi \in \Phi$, the real-valued process $Y[\phi]$ satisfies, $ \sup_{t \geq 0} \abs{\Delta Y_{t}[\phi](\omega)} \leq \rho(\phi)< \infty$ for each $\omega \in \Omega$, thus $Y[\phi]$ has bounded jumps and consequently $Y$ has finite moments to all orders (see \cite{ApplebaumLPSC}, Theorem 2.4.7, p.118-9). Moreover the independent and stationary increments of $Y$ imply that for each $\phi \in \Phi$, the map $t \mapsto \Exp \left( Y_{t}[\phi] \right)$ is additive and measurable. Therefore, there exists some $\goth{m} \in \Phi'_{\beta}$ such that $\Exp \left( Y_{t}[\phi] \right)= t \goth{m} [\phi]$, for all $\phi \in \Phi$, $t  \geq 0$.

Now consider the process $Z =\{ Z_{t} \}_{t\geq 0}$ given by
\begin{equation} \label{auxiliarProcess2LevyItoDecomp}
Z_{t}= Y_{t}-t\goth{m}, \quad \forall \,  t \geq 0. 
\end{equation}  
From the properties of $Y$ and the definition of $\goth{m}$, $Z$ is a $\{ \mathcal{F}_{t} \}$-adapted $\Phi'_{\beta}$-valued, mean-zero, c\`{a}dl\`{a}g, regular L\'{e}vy process with moments to all orders and with jumps satisfying $ \sup_{t \geq 0} \rho'(\Delta Z_{t}(\omega)) \leq 1$ for each $\omega \in \Omega$. 

Now for every $\phi \in \Phi$, let $\kappa(\phi)=\Exp \left[ \abs{ Z_{1}[\phi]}^{2} \right]$. The fact that $Z_{1}$ is a regular random variable with second moments shows that $\kappa$ is a continuous Hilbertian semi-norm on $\Phi$. Moreover the independent and stationary increments of $Z$ imply that $ \Exp \left( \abs{ Z_{t}[\phi]}^{2} \right)= t \kappa(\phi)^{2} $, for all $\phi \in \Phi$, $t  \geq 0$. Hence from Doob's inequality we have for every $T>0$ that: 
\begin{equation} \label{secondMomentProcessZLID}
\Exp \left( \sup_{t \in [0,T]} \abs{Z_{t}[\phi]}^{2} \right) \leq 4 \Exp \left( \abs{ Z_{T}[\phi]}^{2} \right)= 4 T \kappa(\phi)^{2}\, \quad   \forall \,  \phi \in \Phi.
\end{equation}

\begin{theo} \label{wienerPartLevyItoDecomp}
For the $\Phi'_{\beta}$-valued process $X=\{X_{t} \}_{t \geq 0}$ defined by 
\begin{equation} \label{definitionWienerPartLID}
X_{t}=Z_{t}-\int_{B_{\rho'}(1)} f \widetilde{N} (t,df), \quad \forall \, t \geq 0,
\end{equation}
there exist a continuous Hilbertian semi-norm $\eta$ on $\Phi$ and a $\Phi'_{\eta}$-valued $\{ \mathcal{F}_{t} \}$-adapted Wiener process $W=\{ W_{t} \}_{t \geq 0}$ with mean-zero and covariance functional $\mathcal{Q}$ 
(as defined in Theorem \ref{propertiesWienerProcess}) such that $W$ is an indistinguishable version of $X$. Moreover the semi-norm $\eta$ can be chosen such that $\mathcal{Q} \leq K \, \eta$ (for some $K>0$) and the map $i_{\mathcal{Q},\eta}$ is Hilbert-Schmidt. 
\end{theo} 
\begin{prf}
First it is clear that $X$ is a $\Phi'_{\beta}$-valued $\{ \mathcal{F}_{t} \}$-adapted, c\`{a}dl\`{a}g process that has mean-zero and square moments.

Now, we will show that for each $\phi \in \Phi$,  the real-valued process $X[\phi]=\left\{ X_{t}[\phi] \right\}_{t \geq 0}$ is a Wiener process. We proceed in a similar way as in the proof of Proposition 6.2 in \cite{RiedleVanGaans:2009}, where a similar result for the separable Banach space case is considered.

First let $\phi \in \Phi$ be such that $\rho(\phi)=1$. As $Z[\phi]$ defines a real-valued c\`{a}dl\`{a}g L\'{e}vy process it has a corresponding L\'{e}vy-It\^{o} decomposition (see \cite{ApplebaumLPSC}, Theorem 2.4.16, p.126) given by 
$$ Z_{t}[\phi]=b_{\phi}t+\sigma_{\phi}^{2} (B_{\phi})_{t}+ \int_{\{ \abs{y} \leq 1 \}} y \widetilde{N}_{\phi} (t,dy)+  \int_{\{ \abs{y} > 1 \}} y N_{\phi} (t,dy)$$
where $b_{\phi} \in \R$, $\sigma_{\phi}^{2} \in \R_{+}$, $B_{\phi}$ is a standard real-valued Wiener process, $N_{\phi}$ is the Poisson random measure of $Z[\phi]$ and $\widetilde{N}_{\phi}$ its compensated Poisson random measure. All the random components of the decomposition are independent.
For a set $C \in \mathcal{B}(\R)$ that is bounded below we have that  
\begin{equation*}
N_{\phi}(t,C)(\omega)
 =  \sum_{0 \leq s \leq t} \ind{C}{\Delta Z_{s}(\omega)[\phi]} 
=  \sum_{0 \leq s \leq t} \ind{\mathcal{Z}(\phi;C)}{\Delta Z_{s}(\omega)} 
=  N_{Z}\left( t, \mathcal{Z}(\phi;C) \right)(\omega),  
\end{equation*}
where $\mathcal{Z}(\phi;C) \defeq \left\{ f \in \Phi': f[\phi] \in C \right\}$, and $N_{Z}$ denotes the Poisson random measure associated to $Z$. Note that $\mathcal{Z}(\phi;C)$ is a cylindrical set and consequently belongs to $\mathcal{B}(\Phi'_{\beta})$. Moreover as $C$ is bounded below in $\mathcal{B}(\R)$, it follows that $\mathcal{Z}(\phi;C)$ is bounded below in $\mathcal{B}(\Phi'_{\beta})$. To see why this is true, let $\pi_{\phi}$ be given by \eqref{defiMapProjectionCylinder}. Then by \eqref{defiCylindricalSet} and the continuity of $\pi_{\phi}$ it follows that $\overline{\mathcal{Z}(\phi;C)} = \overline{ \pi_{\phi}^{-1}(C)} \subseteq \pi_{\phi}^{-1}(\overline{C})$. Hence if $0 \in \overline{\mathcal{Z}(\phi;C)}$ then $0 \in \pi_{\phi}^{-1}(\overline{C})$, and consequently $0 \in \overline{C}$. But this contradicts the fact that $C$ is bounded below. Therefore, $\mathcal{Z}(\phi;C)$ is bounded below. 

Now let $C=[-1,1]^{c}$ and $D=\left\{ f \in \Phi': \abs{f[\phi]} \leq 1 \right\}$. We then have that $D=\mathcal{Z}(\phi;C)^{c}$ and because $\phi \in B_{\rho}(1)$, it follows that $B_{\rho'}(1) \subseteq D$. 
Now because the jumps of $Z$ satisfy $\sup_{t \geq 0} \rho'(\Delta Z_{t}(\omega)) \leq 1$ for each $\omega \in \Omega$, the support of $N_{Z}(t, \cdot)$ is in $B_{\rho'}(1)$ for each $t\geq 0$, and consequently the support of $\widetilde{N}_{Z}(t,\cdot)$ is also in $B_{\rho}(1)$ for $t \geq 0$. Since $B_{\rho'}(1) \subseteq D$, it follows that 
\begin{equation*}
\int_{D} f \widetilde{N}_{Z} (t,df)[\phi] 
 = \int_{B_{\rho'}(1)} f \widetilde{N}_{Z} (t,df)[\phi]+  \int_{D \setminus B_{\rho'}(1)} f \widetilde{N}_{Z} (t,df)[\phi] 
=  \int_{B_{\rho'}(1)} f \widetilde{N}_{Z} (t,df)[\phi]
\end{equation*}
and 
$$\int_{D^{c}} f N_{Z} (t,df)[\phi]=0.$$
Moreover $\widetilde{N}_{Z}$ coincides with $\widetilde{N}$ in $B_{\rho'}(1)$, so we have that 
\begin{eqnarray*}
Z_{t}[\phi]
& = & b_{\phi}t+\sigma_{\phi}^{2} (B_{\phi})_{t}+ \int_{\left\{ \abs{y} < 1 \right\}} y \widetilde{N}_{\phi} (t,dy)+  \int_{\left\{ \abs{y} \geq 1 \right\}} y N_{\phi} (t,dy) \\
& = & b_{\phi}t+\sigma_{\phi}^{2} (B_{\phi})_{t}+ \int_{D} f \widetilde{N}_{Z} (t,df)[\phi]+  \int_{D^{c}} f N_{Z} (t,df)[\phi] \\
& = & b_{\phi}t+\sigma_{\phi}^{2} (B_{\phi})_{t}+ \int_{B_{\rho'}(1)} f \widetilde{N}_{Z} (t,df)[\phi] \\
& = & b_{\phi}t+\sigma_{\phi}^{2} (B_{\phi})_{t}+ \int_{B_{\rho'}(1)} f \widetilde{N} (t,df)[\phi] \\
\end{eqnarray*}
Now taking expectations we obtain that for every $t \geq 0$,
$$0=\Exp Z_{t}[\phi] = b_{\phi}t  +\sigma_{\phi}^{2} \Exp \left( (B_{\phi})_{t} \right) + \Exp \left( \int_{ B_{\rho'}(1)} f \widetilde{N} (t,df)[\phi] \right) = b_{\phi}t$$
consequently $b_{\phi}=0$. We obtain $X_{t}[\phi]=Z_{t}[\phi]-\int_{ B_{\rho'}(1) } f \widetilde{N} (t,df)[\phi]=\sigma_{\phi}^{2} (B_{\phi})_{t}$ and so $X[\phi]$ is a Wiener process. 
The same representation holds for arbitrary $\phi \in \Phi$, as can be seen by replacing $\phi$ with $\phi / \rho(\phi)$ in the argument just given. Therefore $X[\phi]$ is a Wiener process $\forall \phi \in \Phi$. 

Now note that for every $T>0$ and $\phi \in \Phi$, from Doob's inequality, \eqref{secondMomentMartingaleLID} and \eqref{secondMomentProcessZLID}, we have that
\begin{eqnarray*}
 \Exp \left( \sup_{t \in [0,T]} \abs{X_{t}[\phi]}^{2} \right) 
 & \leq & 4\Exp \left( \abs{ X_{T}[\phi]}^{2} \right) \\
& \leq & 8 T \left( \Exp \left( \abs{ Z_{T}[\phi]}^{2} \right) + \Exp \left( \abs{ M_{T}[\phi]}^{2} \right) \right)  \\
& \leq & 8T (\kappa(\phi)^{2} + C_{\rho} \, q(\phi)^{2}),
\end{eqnarray*}
where $C_{\rho}= \int_{ B_{\rho'}(1)} q'(f)^{2} \nu(df)<\infty$. Let $\sigma$ be a continuous Hilbertian semi-norm on $\Phi$ such that $\kappa \leq \sigma$ and $q \leq \sigma$. Then from the above inequalities for each $T>0$ and $\phi \in \Phi$ we have
$$ \Exp \left( \sup_{t \in [0,T]} \abs{X_{t}[\phi]}^{2} \right) 
  \leq  8T (1+ C_{\rho} ) \sigma(\phi)^{2}. $$
Then Theorem \ref{theoCylindrLevyHilbContHaveLevyCadlagVersHilbSpace} shows that there exists a continuous Hilbertian semi-norm $\eta$ on $\Phi$, $\sigma \leq \eta$, such that $i_{\sigma,\eta}$ is Hilbert-Schmidt and there exists a $\Phi'_{\eta}$-valued Wiener processes (i.e. a continuous L\'{e}vy process) $W=\{ W_{t}\}_{t \geq 0}$ that has finite second moments in $\Phi'_{\eta}$ and such that for every $\phi \in \Phi$, $W[\phi]= \{ W_{t}[\phi] \}_{t \geq 0}$ is a version of $X[\phi]= \{ X_{t}[\phi] \}_{t \geq 0}$. However as both $W$ and $X$ are regular c\`{a}dl\`{a}g processes in $\Phi'_{\beta}$, then the fact that $W[\phi]=X[\phi]$ for each $\phi \in \Phi$ implies that $W$ and $X$ are indistinguishable (Proposition \ref{propCondiIndistingProcess}). 
Hence, $W$ is $\{ \mathcal{F}_{t} \}$-adapted and is also $\Phi'_{\beta}$-valued Wiener process. 

Finally if $\mathcal{Q}$ is the covariance functional of $W$, from \eqref{covarianceFunctWienerProcess} it follows that for every $\phi \in \Phi$ we have
$$ \mathcal{Q}(\phi)^{2}=\Exp \left( \abs{ W_{1}[\phi]}^{2} \right)=  \Exp \left( \abs{ X_{1}[\phi]}^{2} \right) \leq 2 (1+ C_{\rho} ) \sigma(\phi)^{2} \leq 2 (1+ C_{\rho} ) \eta(\phi)^{2}. $$
Then $\mathcal{Q} \leq K \, \eta$ with $K^{2}=2 (1+ C_{\rho} )$. Moreover, because $i_{\mathcal{Q},\sigma}$ is linear and continuous and $i_{\sigma,\eta}$ is Hilbert-Schmidt, we have that $i_{\mathcal{Q},\eta}= i_{\sigma,\eta} \circ i_{\mathcal{Q},\sigma}$ is Hilbert-Schmidt. 
\end{prf}

We are ready for the main result of this section. 

\begin{theo}[L\'{e}vy-It\^{o} decomposition] \label{levyItoDecompositionTheorem}
Let $L=\left\{ L_{t} \right\}_{t\geq 0}$ be a $\Phi'_{\beta}$-valued L\'{e}vy process. Then for each $t \geq 0$ it has the following representation
\begin{equation} \label{levyItoDecomposition}
L_{t}=t\goth{m}+W_{t}+\int_{B_{\rho'}(1)} f \widetilde{N} (t,df)+\int_{B_{\rho'}(1)^{c}} f N (t,df)
\end{equation}
where 
\begin{enumerate}
\item $\goth{m} \in \Phi'_{\beta}$, 
\item $\rho$ is a continuous Hilbertian semi-norm on $\Phi$ such that the L\'{e}vy measure $\nu$ of $L$ satisfies \eqref{integrabilityLevyMeasureLevyProcess} and $B_{\rho'}(1) \defeq \{f \in \Phi'_{\beta}: \rho'(f) \leq 1\} \subseteq \Phi'_{\beta}$ is bounded, closed, convex and balanced, 
\item $\{ W_{t} \}_{t \geq 0}$ is a $\Phi'_{\eta}$-valued Wiener process with mean-zero and covariance functional $\mathcal{Q}$, where $\eta$ is a continuous Hilbertian semi-norm on $\Phi$ such that $\mathcal{Q} \leq K \eta$ (for some $K>0$) and the map $i_{\mathcal{Q},\eta}$ is Hilbert-Schmidt,
\item $\left\{ \int_{B_{\rho'}(1)} f \widetilde{N} (t,df): t\geq 0 \right\}$ is a $\Phi'_{q}$-valued mean-zero, square integrable, c\`{a}dl\`{a}g L\'{e}vy process with characteristic function given by \eqref{charactFunctionMartingaleLevyItoDecomp} and second moments given by \eqref{secondMomentMartingaleLID}, where $q$ is a continuous Hilbertian semi-norm on $\Phi$ such that $\rho \leq q$ and the map $i_{\rho,q}$ is Hilbert-Schmidt, 
\item $\left\{ \int_{B_{\rho'}(1)^{c}} f N (t,df): t\geq 0 \right\}$ is a $\Phi'_{\beta}$-valued c\`{a}dl\`{a}g L\'{e}vy process defined in \eqref{definitionCompoundPoissonLID} by means of a Poisson integral with respect to the Poisson random measure $N$ of $L$ on the set $B_{\rho'}(1)^{c}$ and  with characteristic function given by \eqref{charactFunctionCompoundPoissonLID}. 
\end{enumerate}
All the random components of the decomposition \eqref{levyItoDecomposition} are independent.     
\end{theo}
\begin{prf}
The decomposition \eqref{levyItoDecomposition} and the properties of its components follow from Theorems \ref{existenceMartingaleLevyItoDecomp} and \ref{wienerPartLevyItoDecomp}, Proposition \ref{propExistenceLargeJumpPartLID}, \eqref{auxiliarProcess1LevyItoDecomp} and \eqref{auxiliarProcess2LevyItoDecomp}. Now we prove the independence of the components in \eqref{levyItoDecomposition}. 

For any $\phi_{1}, \dots, \phi_{n} \in \Phi$, by considering the \Levy-\Ito{} decomposition of the $\R^{n}$-valued \Levy{} process $ \left\{ \left( L_{t}[\phi_{1}],\dots, L_{t}[\phi_{n}] \right) \right\}_{t \geq 0}$, it follows that the $\R^{n}$-valued processes \newline  
$ \{ (W_{t}[\phi_{1}], \dots, W_{t}[\phi_{n}]) \}_{t\geq 0}$, $\left\{ \left( \int_{B_{\rho'}(1)} f \widetilde{N} (t,df)[\phi_{1}], \dots, \int_{B_{\rho'}(1)} f \widetilde{N} (t,df)[\phi_{n}] \right): t\geq 0 \right\}$, \newline
and $\left\{ \left( \int_{B_{\rho'}(1)^{c}} f N (t,df)[\phi_{1}], \dots, \int_{B_{\rho'}(1)^{c}} f N (t,df)[\phi_{n}] \right): t\geq 0 \right\}$ are independent. But because the processes $\{ W_{t} \}_{t \geq 0}$, $\left\{\int_{B_{\rho'}(1)} f \widetilde{N} (t,df): t\geq 0 \right\}$ and $\left\{ \int_{B_{\rho'}(1)} f N (t,df): t\geq 0 \right\}$ are regular, then Proposition \ref{independenceStochProcessDualSpace} shows that they  are independent. 
\end{prf}

As an important by-product of the proof of the L\'{e}vy-It\^{o} decomposition we obtain a L\'{e}vy-Khintchine theorem for the characteristic function of any $\Phi'_{\beta}$-valued L\'{e}vy process.

\begin{theo}[L\'{e}vy-Khintchine theorem for $\Phi'_{\beta}$-valued L\'{e}vy processes]  \label{levyKhintchineFormulaLevyProcessTheorem} \hfill
\begin{enumerate}
\item If $L=\left\{ L_{t} \right\}_{t\geq 0}$ is a $\Phi'_{\beta}$-valued, regular, c\`{a}dl\`{a}g L\'{e}vy process, there exist $\goth{m} \in \Phi'_{\beta}$, a continuous Hilbertian semi-norm $\mathcal{Q}$ on $\Phi$, a L\'{e}vy measure $\nu$ on $\Phi'_{\beta}$ and a continuous Hilbertian semi-norm $\rho$ on $\Phi$ for which $\nu$ satisfies \eqref{integrabilityLevyMeasureLevyProcess}; and such that for each $t \geq 0$, $\phi \in \Phi$, 
\begin{equation} \label{levyKhintchineFormulaLevyProcess}
\begin{split}
& \Exp \left( e^{i L_{t}[\phi] } \right) = e^{t \eta(\phi)}, \quad  \mbox{ with} \\ 
& \eta(\phi)= i \goth{m}[\phi] - \frac{1}{2} \mathcal{Q}(\phi)^{2} + \int_{\Phi'_{\beta}} \left( e^{i f[\phi]} -1 - i f[\phi] \ind{ B_{\rho'}(1)}{f} \right) \nu(d f).  
\end{split}
\end{equation}

\item Conversely, let $\goth{m} \in \Phi'_{\beta}$, $\mathcal{Q}$ be a continuous Hilbertian semi-norm on $\Phi$, and $\nu$ be a $\theta$-regular L\'{e}vy measure on $\Phi'_{\beta}$ satisfying \eqref{integrabilityLevyMeasureLevyProcess} for a continuous Hilbertian semi-norm $\rho$ on $\Phi$. There exists a $\Phi'_{\beta}$-valued, regular, c\`{a}dl\`{a}g L\'{e}vy process $L=\left\{ L_{t} \right\}_{t\geq 0}$ defined on some probability space $\ProbSpace$, unique up to equivalence in distribution, whose characteristic function is given by \eqref{levyKhintchineFormulaLevyProcess}. In particular, $\nu$ is the L\'{e}vy measure  of $L$. 
\end{enumerate}
\end{theo}
\begin{prf} If $L$ is a $\Phi'_{\beta}$-valued, regular, c\`{a}dl\`{a}g L\'{e}vy process then \eqref{levyKhintchineFormulaLevyProcess} follows from the independence of the random components of the decomposition \eqref{levyItoDecomposition}, \eqref{charactFunctionWienerProcess} (recall here that $W$ has mean zero and covariance functional $\mathcal{Q}$), \eqref{charactFunctionMartingaleLevyItoDecomp} and \eqref{charactFunctionCompoundPoissonLID}. 

For the converse, assume we have $\goth{m}$, $\mathcal{Q}$, $\nu$ and $\rho$ with the properties in the statement of the theorem. 
First as $\nu$ is a $\sigma$-finite Borel measure on $\Phi'_{\beta}$ (Proposition \ref{propLevyMeasuresAreRadon}), there exist a stationary Poisson point processes $p= \{ p(t) \}_{t \geq 0}$ on $(\Phi'_{\beta},\mathcal{B}(\Phi'_{\beta}))$ with associated Poisson random measure $R$, $p$ and $R$ unique up to equivalence in distribution, such that $\nu$ is the characteristic measure of $p$ (see \cite{IkedaWatanabe}, Theorem I.9.1, p.44; see also  \cite{Sato}, Proposition 19.4, p.122). If $ U_{n} \in \mathcal{B}(\Phi'_{\beta})$, for $n \in \N$, are disjoint, $\Phi'_{\beta}=\bigcup_{n} U_{n}$ and $\nu(U_{n})< \infty$ for every $n \in \N$, the point process $p$ can be constructed from a sequence of stopping times $\tau^{(n)}_{i}$ with exponential distribution with parameter $\nu(U_{n})$ and a sequence $\xi^{(n)}_{i}$ of $\Phi'_{\beta}$-valued random variables with probability  distribution $\nu(\cdot)/\nu(U_{n})$ 
(see details in \cite{IkedaWatanabe}, Theorem I.9.1, p.44). Because $\nu$ is concentrated on $\Phi'_{\theta}$ for a weaker countably Hilbertian topology $\theta$ on $\Phi$ (Lemma \ref{lemmRestrLevyMeasureIsRadon}), it follows  that the random variables $\xi^{(n)}_{i}$ are regular. But as $p$ takes the values of these random variables (indeed we have  $p(\tau^{(n)}_{1}+ \dots + \tau^{(n)}_{i})=\xi^{(n)}_{i}$ for $n,i \in \N$), then $p$ is a regular process in $\Phi'_{\beta}$.       

Now note that in the proof of Theorem \ref{existenceMartingaleLevyItoDecomp}, we only used the fact that the L\'{e}vy measure $\nu$ of a L\'{e}vy process $L$ satisfies the integrability condition in  \eqref{integrabilityLevyMeasureLevyProcess}, and that the Poisson integral with respect to the Poisson random measure $N$ of $L$ exists and satisfies the properties given in Section \ref{subsectionPRMPI}. Since we can define Poisson integrals with respect to the Poisson random measure $R$ of $p$ satisfying the properties given in Section \ref{subsectionPRMPI} (here we use that $p$ is a regular process), and $\nu$ satisfies \eqref{integrabilityLevyMeasureLevyProcess}, we can replicate the arguments in the proof of Theorem \ref{existenceMartingaleLevyItoDecomp} to conclude that there exists a continuous Hilbertian semi-norm $q$ on $\Phi$ such that $\rho \leq q$ and the map $i_{\rho,q}$ is Hilbert-Schmidt, and a $\Phi'_{q}$-valued mean-zero, square integrable, c\`{a}dl\`{a}g L\'{e}vy process $\widetilde{M}=\{ \widetilde{M}_{t}\}_{t \geq 0}$ with characteristic function given by \eqref{charactFunctionMartingaleLevyItoDecomp}. 

On the other hand, because from \eqref{integrabilityLevyMeasureLevyProcess}  we have $\nu(B_{\rho'}(1)^{c})<\infty$, it follows from Proposition \ref{propExistenceLargeJumpPartLID} that there exists a
$\Phi'_{\beta}$-valued, regular, c\`{a}dl\`{a}g L\'{e}vy process 
 $\widetilde{J} = \{ \widetilde{J}_{t} \}_{t \geq 0}$, where $\widetilde{J}_{t}=\int_{B_{\rho'}(1)^{c}} f R(t,df)$ as given in \eqref{definitionCompoundPoissonLID} (with $N$ replaced by $R$), with characteristic function  \eqref{charactFunctionCompoundPoissonLID}. 
Moreover, from Theorem \ref{existenceWienerProcess} there exists a $\Phi'_{\beta}$-valued Wiener process $\widetilde{W}= \{\widetilde{W}_{t} \}_{t \geq 0}$, unique up to equivalence in distribution, such that $\goth{m}$ and $\mathcal{Q}$ are the mean and the covariance functional of $\widetilde{W}$. Hence, $\widetilde{W}$ has characteristic function given by \eqref{charactFunctionWienerProcess}. 

We can assume without loss of generality that $ \widetilde{W}$, $\widetilde{M}$ and $ \widetilde{J}$ are independent $\Phi'_{\beta}$-valued process defined on some probability space $\ProbSpace$ (see e.g. \cite{Kallenberg}, Corollary 6.18, p.117).  Hence, if we define $L=\left\{ L_{t} \right\}_{t\geq 0}$, where for each $t \geq 0$, $ L_{t}= \widetilde{W}_{t}+ \widetilde{M}_{t}+\widetilde{J}_{t}$, then $L$ being the sum of a finite number of independent c\`{a}dl\`{a}g L\'{e}vy process is also a $\Phi'_{\beta}$-valued, c\`{a}dl\`{a}g L\'{e}vy process. It is also unique up to equivalence in distribution, and for each $t \geq 0$, $L_{t}$ has characteristic function given by \eqref{levyKhintchineFormulaLevyProcess}. 
\end{prf}

\section{L\'{e}vy-Khintchine theorem for infinitely divisible measures} \label{sectionLKTIDM}

Our final result is the L\'{e}vy-Khintchine formula for infinitely divisible measures in the dual of a nuclear space. This result is not covered by the work of Dettweiler \cite{Dettweiler:1976} (see Section \ref{sectionExampCommen}).

\begin{theo}[L\'{e}vy-Khintchine theorem] \label{levyKhintchineFormula} 
Let $\mu \in \goth{M}_{R}^{1}(\Phi'_{\beta})$. Then:
\begin{enumerate}
\item If $\Phi$ is also a barrelled space and if $\mu$ is  infinitely divisible, then there exists $\goth{m} \in \Phi'_{\beta}$, a continuous Hilbertian semi-norm $\mathcal{Q}$ on $\Phi$, a L\'{e}vy measure $\nu$ on $\Phi'_{\beta}$ and a continuous Hilbertian semi-norm $\rho$ on $\Phi$ for which $\nu$ satisfies \eqref{integrabilityLevyMeasureLevyProcess}; such that the characteristic function of $\mu$ satisfies the following formula for every $\phi \in \Phi$: 
\begin{equation} \label{levyKhintchineFormulaEquation}
\widehat{\mu}(\phi)=\exp \left[ i \goth{m}[\phi] - \frac{1}{2} \mathcal{Q}(\phi)^{2} + \int_{\Phi'_{\beta}} \left( e^{i f[\phi]} -1 - i f[\phi] \ind{ B_{\rho'}(1)}{f} \right) \nu(d f) \right].
\end{equation}
\item Conversely, let $\goth{m} \in \Phi'_{\beta}$, $\mathcal{Q}$ be a continuous Hilbertian semi-norm on $\Phi$, and $\nu$ be a $\theta$-regular L\'{e}vy measure on $\Phi'_{\beta}$ satisfying \eqref{integrabilityLevyMeasureLevyProcess} for a continuous Hilbertian semi-norm $\rho$ on $\Phi$. If $\mu$ has characteristic function given by \eqref{levyKhintchineFormulaEquation}, then $\mu$ is infinitely divisible. 
\end{enumerate}
\end{theo}
\begin{prf}
First suppose that $\mu$ is infinitely divisible. Then, it follows from Theorem \ref{theoInfiDivisMeasuImpliLevyProc} that there exists a $\Phi'_{\beta}$-valued, regular, c\`{a}dl\`{a}g L\'{e}vy process $L=\left\{ L_{t} \right\}_{t\geq 0}$ such that $\mu_{L_{1}}=\mu$. Then the existence of $\mu$, $\mathcal{Q}$, $\nu$ and $\rho$ follows from Theorem \ref{levyKhintchineFormulaLevyProcessTheorem}(1). Furthermore the fact that $\mu$ satisfies \eqref{levyKhintchineFormulaEquation} follows from taking $t=1$ in \eqref{levyKhintchineFormulaLevyProcess} and because $\mu_{L_{1}}=\mu$. 

Conversely, suppose that $\mu$ satisfies \eqref{levyKhintchineFormulaEquation} for the given $\mu$, $\mathcal{Q}$, $\nu$ and $\rho$. Then it follows from Theorem \ref{levyKhintchineFormulaLevyProcessTheorem}(2) that there exists a $\Phi'_{\beta}$-valued, regular, c\`{a}dl\`{a}g L\'{e}vy process $L=\left\{ L_{t} \right\}_{t\geq 0}$ such that $\mu_{L_{1}}=\mu$. But then Theorem \ref{theoLevyDefinesConvSemig} shows that $\mu$ is infinitely divisible. 
\end{prf}

\begin{rema}
If $\Phi$ is a barrelled nuclear space, the assumption that the L\'{e}vy measure $\nu$ is $\theta$-regular in Theorems \ref{levyKhintchineFormulaLevyProcessTheorem}(2) and  \ref{levyKhintchineFormula}(2) can be disposed because every L\'{e}vy measure on $\Phi$ is $\theta$-regular (see Proposition  \ref{propLevyMeasuresAreRadon}). 
\end{rema}

\section{Examples, Remarks and Comparisons} \label{sectionExampCommen}

All throughout this paper we have considered L\'{e}vy processes and infinitely divisible measures defined on the dual of a nuclear space. In this section we provide examples of nuclear spaces arrising in applications and give emphasis on examples which possess the additional structure of barrelledness.  Furthermore, we provide some additional remarks on the results obtained throught this paper and we also make comparison of our results with those obtained by other authors.

\textbf{Examples:} Many spaces of functions widely used in analysis are examples of (complete) barrelled nuclear spaces. We can mention for example  (for details see \cite{Pietsch, Schaefer, Treves}) the spaces $\mathscr{E}_{K} \defeq \mathcal{C}^{\infty}(K)$ ($K$: compact subset of $\R^{d}$), $\mathscr{E}\defeq \mathcal{C}^{\infty}(\R^{d})$, the rapidly decreasing functions $\mathscr{S}(\R^{d})$, the space of harmonic functions $\mathcal{H}(U)$ ($U$: open subset of $\R^{d}$), and the space of test functions $\mathscr{D}(U) \defeq \mathcal{C}_{c}^{\infty}(U)$ ($U$: open subset of $\R^{d}$). Their (strong) dual spaces $\mathscr{E}'_{K}$, $\mathscr{E}'$, $\mathscr{S}'(\R^{d})$, $\mathcal{H}'(U)$, $\mathscr{D}'(U)$ are also (complete) barrelled nuclear spaces. Other examples are the space of polynomials $\mathcal{P}_{n}$ in $n$-variables, the space of real-valued sequences $\R^{\N}$ (with direct sum topology), and the space of real-analytic functions $\mathcal{A}(V)$ ($V$: closed subset of $\R^{d}$, see \cite{HogbeNlendMoscatelli}). 

There are interesting examples of nuclear spaces that are not (or might not be)   barrelled. As examples we have the space $\R^{D}$ equipped with its product topology ($D$ arbitrary set, see \cite{Treves}) and the nuclear K\"{o}the sequence spaces (see \cite{HogbeNlendMoscatelli, Jarchow}). Other examples are the space of holomorphic functions defined on a  (quasi-)complete dual nuclear space (e.g. $\mathcal{H}(\mathscr{D}'(\R^{d}))$, see \cite{Dineen}), the space of continuous linear operators between a semi-reflexive dual nuclear space into a nuclear space (e.g. the space $\mathscr{D}'(U; \R^{\N}) \defeq \mathcal{L}(\mathcal{C}^{\infty}_{c}(U),\R^{\N})$ of distributions with values in $\R^{\N}$, $U \subseteq \R^{n}$ open), and tensor products of  nuclear spaces (e.g. the non-barrelled nuclear space $\mathscr{D}(\R) \widehat{\otimes} \mathscr{S}(\R^{d}) $, or the the space of holomorphic functions with values in the space of distributions $\mathcal{H}(U;\mathscr{D}'(\R^{d})) \cong \mathcal{H}(U) \widehat{\otimes} \mathscr{D}'(\R^{d})$, $U \subseteq \R^{n}$ open); for references see \cite{Schaefer, Treves}. A somewhat more abstract example of a (complete) nuclear space that is not barrelled is the following: if $E$ is an infinite dimensional Banach space, it is possible to show that $E$ is the strong dual of some complete nuclear space $\Phi$ that is not barrelled (see Corollary 1 of Theorem IV.4.3.3 in  \cite{HogbeNlendMoscatelli}). 


 
\textbf{Remarks and comparisons:}
The L\'{e}vy-Khintchine formula for infinitely divisible measures defined on a separable Hilbert space  was proven by Varadhan in \cite{Varadhan:1962}.  In \cite{Fernique:1967}, Fernique proved the existence of a L\'{e}vy-Khintchine formula for infinitely divisible measures defined on $\mathcal{D}'$. Later, in  \cite{Tortrat:1967}, Tortrat studied the representation  of an infinitely divisible measure defined on a topological vector spaces as the convolution of a Gaussian measure and a ``generalized'' Poisson measure, under different additional assumptions on the underlying space. This study was continued by the same author in \cite{Tortrat:1969}, moreover, a L\'{e}vy-Khintchine formula was proven for infinitely divisible measures defined on a locally convex space that are tight with respect to a family of compact Hilbertian subsets. In particular, the result of Tortrat extended the previous work of Varadhan and Fernique. 

However, it was Dettweiler in \cite{Dettweiler:1976} who proved the existence of a L\'{e}vy-Khintchine formula for infinitely divisible measures defined on a complete locally convex space. There, Dettweiler proved that, for such a general context, the L\'{e}vy measure of an infinitely divisible measure no longer possesses the square integrability condition with respect to a Hilbertian seminorm obtained by Varadhan and Fernique. Nevertheless, Dettweiler showed that if the underlying locally convex space is a complete Badrikian space (this last property means that the space possesses a fundamental system of Hilbertian, compact, convex, balanced subsets; such a class of spaces was introduced in \cite{Badrikian}, Expos\'{e} no.10), then a characterization of L\'{e}vy measures in terms of square integrability with respect to a Hilbertian seminorm also holds. The result of Dettweiler hence contains those of Varadhan, Fernique and Tortrat.    

The L\'{e}vy-Khintchine formula obtained by Dettweiler on a complete Badrikian space 
coincides with \eqref{levyKhintchineFormulaEquation}. However, here it is very important to stress the fact that in Theorem \ref{levyKhintchineFormula}, the space $\Phi'_{\beta}$ is not necessary complete nor a Badrikian space. Hence we have obtained a L\'{e}vy-Khintchine formula that works in a case that is not covered by the formula proved by Dettweiler in \cite{Dettweiler:1976}. Similarly, we have been able to show that the L\'{e}vy measure of an infinitely divisible measure defined on $\Phi'_{\beta}$ is characterized by \eqref{integrabilityLevyMeasureLevyProcess}, and hence, such a characterization is possible beyond the context of complete Badrikian spaces. 

It is interesting the fact that our proof of the L\'{e}vy-Khintchine formula was carried out via the connection between infinitely divisible measures, convolution semigroups of measures and L\'{e}vy processes. This connection is very well known if the space is finite dimensional (see e.g. \cite{ApplebaumLPSC, Sato}), and can be extended to a complete Souslin locally convex space using for example the results of Schwartz in \cite{Schwartz:1977} (see \cite{Baumgartner:2015}). Observe that in our context the fact that an infinitely divisible measure can be embedded in a continuous time convolution semigroup follows from arguments essentially due to Siebert \cite{Siebert:1974, Siebert:1976}, that are of very general nature. However, to the extent of our knowledge our Theorem \ref{theoInfiDivisMeasuImpliLevyProc} is the first that shows the correspondence of a L\'{e}vy process with a given infinitely divisible measure on the strong dual $\Phi'_{\beta}$ of a barrelled nuclear space $\Phi$ (observe that the space $\Phi'_{\beta}$ is in general not a Souslin space, hence we have a situation that is not contained in the work \cite{Schwartz:1977}). Here, we want to stress the fact that the role played by our assumption of nuclearity of the space $\Phi$ is present in the use of the regularization theorem for cylindrical L\'{e}vy processes (Theorem \ref{theoCylindrLevyProcessHaveLevyCadlagVersion}). Such result result does not work in general for other classes of locally convex spaces (e.g. in Hilbert spaces) without impossing additional assumptions on equicontinuity on a weaker topology (e.g. the Sazonov's topology). Moreover, the assumption of barrelledness is based on the fact that it implies that uniform tightness of a collection of measures implies the equicontinuity of its characteristic functions, and this last condition is essential for the use of the regularization theorem.  

It is a very interesting fact that if one consider specific classes of infinitely divisible measures, for example, the Poisson and the stable measures, one can obtain a L\'{e}vy-Khintchine formula and some integrability characterizations for its L\'{e}vy measures for general (complete) locally convex spaces (see e.g. \cite{ChungRajputTortrat:1982, Rajput:1977, Tortrat:1976, Tortrat:1977}). The same degree of  generality on the space can be obtained if one considers infinitely divisible measures on cones (see e.g. \cite{Dettweiler:1976-1, PerezAbreuRosinski:2007}). 


The existence of c\`{a}dl\`{a}g versions and the  L\'{e}vy-It\^{o} decomposition for L\'{e}vy processes taking values in the dual of (certain classes of) nuclear spaces have been considered firstly by \"{U}st\"{u}nel in \cite{Ustunel:1984}, and more recently by Baumgartner in \cite{Baumgartner:2015}. In \cite{Ustunel:1984} it is assumed that $\Phi$ is a complete, separable, nuclear space (and barrelled; this is not explicitly assumed but is used in the main arguments) and that $\Phi'_{\beta}$ is nuclear and Souslin. In  \cite{Baumgartner:2015}, the assumptions are that the space where the L\'{e}vy processes takes values is a complete, Souslin, locally convex space which possesses  a fundamental system of Banach, compact, convex, separable, subsets (as for example a Badrikian space). In the context of duals of nuclear spaces, the assumptions in \cite{Baumgartner:2015} reduce to those in \cite{Ustunel:1984}. 

Observe that for our Theorems \ref{theoConditCadlagVersionLevyProc} and  \ref{levyItoDecompositionTheorem} on the existence of c\`{a}dl\`{a}g versions and the L\'{e}vy-It\^{o} decomposition, we have not assumed any property on the space $\Phi$ other than being nuclear, and we have made no assumptions on its dual space $\Phi'_{\beta}$. We have made only assumptions on the L\'{e}vy process (see Assumptions \ref{generalAssumptionsLevyProcess}) but these are always satisfied if $\Phi$ is additionally barrelled. Hence, our results generalize those in \cite{Baumgartner:2015, Ustunel:1984}. Indeed, we have obtained a more detailed description of the components in the decomposition. Furthermore, in both works \cite{Baumgartner:2015, Ustunel:1984} the authors used the L\'{e}vy-Khintchine formula of Dettweiler as a fundamental input, but in our case, we have obtained this formula as a by-product. Finally, it is worth to mentioning that, to the extent of our knowledge, this work is the first that considers regularization of cylindrical L\'{e}vy processes in duals of nuclear spaces. The result in Theorem \ref{theoCylSemiGroupImpliesCylLevyProc} that relates some families of cylindrical measures  with a cylindrical L\'{e}vy process is also new, and it is very interesting from the point of view that it holds for general locally convex spaces. 
    
\textbf{Acknowledgements} { The author would like to thank  David Applebaum for all his helpful comments and suggestions. 
Thanks also to the The University of Costa Rica for providing financial support through the grant 820-B6-202	``Ecuaciones diferenciales en derivadas parciales en espacios de dimensi\'{o}n infinita''. Some earlier parts of this work were carried out at The University of Sheffield and the author wishes to express his gratitude. Many thanks are also due to the referees, who made very helpful remarks.
}

\end{document}